\documentclass{article}
\usepackage{amsfonts}
\usepackage{amsmath}
\usepackage{graphicx}
\usepackage{xcolor}
\usepackage{subcaption}
\usepackage{ulem}
\usepackage{natbib}

\usepackage[a4paper, left=3cm, right=3cm, top=4cm, bottom=4cm]{geometry}

\newtheorem{theorem}{Theorem}

\newtheorem{definition}[theorem]{Definition}
\newtheorem{proposition}[theorem]{Proposition}

\newenvironment{proof}[1][Proof]{\noindent \textbf{#1.} }{\  \rule{0.5em}{0.5em}}

\begin{document}

\title{\textbf{Robust tests for log-logistic models based on minimum density power divergence estimators}}
\date{}
\date{}
\author{Felipe, A., Jaenada, M., Miranda, P. and Pardo, L.}
\maketitle

\begin{abstract}
The log-logistic distribution is a versatile parametric family widely used across various applied fields, including survival analysis, reliability engineering, and econometrics.  When estimating parameters of the log-logistic distribution, hypothesis testing is necessary to verify assumptions about these parameters. The Wald test and Rao test provide formal methods for testing hypotheses about these parameters. However, these test statistics are not robust, and their rejection decisions may be affected by data contamination. In this paper we develop new families of Wald-type  test statistics and Rao-type test statistics based on minimum density power divergence estimators (MDPDEs) for the parameters of the log-logistic distribution. These new families generalize the Wald and Rao test statistics,  inheriting the robustness properties from the MDPDEs and thus addressing the lack of robustness of the classical tests.  Explicit expressions for the test statistics under the log-logistic model for both simple and composite null hypotheses are derived, and their properties are analyzed in detail. An extensive simulation study empirically demonstrates  the robustness of these families and compares their performance with the classical methods.
\end{abstract}


\section{Introduction \label{sec:1}}

The log-logistic distribution is a continuous probability distribution commonly used in survival analysis, reliability engineering, and econometrics. Indeed, the Fisk distribution, a special case of the log-logistic family, is widely used in economics to model income distribution, wealth distribution, and price variations.  The log-logistic distribution provides a flexible model for right-skewed data, which is common in economic studies where a small proportion of individuals or entities hold a large share of wealth or income. 
Further, the log-logistic distribution is also useful in modeling lifetime data, where the hazard function exhibits a non-monotonic shape, allowing for an initial increase followed by a decrease over time.

The log-logistic distribution can be defined as the logarithmic transformation of a logistic distribution, similar to how the log-normal distribution is derived from a normal distribution (see \cite{bai74}). The probability density function (pdf) of the log-logistic distribution is given by
\begin{equation}
f_{\alpha ,\beta }(x)=\frac{\beta \alpha ^{\beta }x^{\beta -1}}{\left( x^{\beta }+\alpha ^{\beta }\right) ^{2}},\text{ \qquad }x>0\text{ \ }(\alpha ,\beta >0)  \label{W1.1}
\end{equation}
where $\alpha >0$ is a scale parameter and  $\beta >0$ is a shape parameter. For the log-logistic distribution, the scale $\alpha$ represents the median of the distribution. The distribution is unimodal for $\beta >1.$ Therefore, the parameter space is restricted to 
$
	\Theta =\left \{ \left( \alpha ,\beta \right) /\alpha ,\beta >0\right \} = \mathbb{R}^{+}\times \mathbb{R}^{+}.
$

Moreover, the log-logistic distribution is a special case of Burr's type-XII distribution (see \cite{bur42}), and also a special case of the ``Kappa distributions'' (see \cite{mijo73}) widely used in hydrology, meteorology, actuarial science. The log-logistic distribution's properties and characteristics were first studied in \cite{shda63}, where its flexibility and suitability for modeling right-skewed data was evidenced. For more detail about this distribution, we refer to the works of \cite{tajo82} and \cite{bama87}.

Due to its strong suitability for modeling right-skewed data, it has been widely applied across various fields, including Hydrology, Economics, Survival, and Quality Control, among many others. For instance,  \cite{shmitr88} applied the log-logistic distribution to model precipitation data from some Canadian areas. In \cite{asma03}  the superior fit of the log-logistic distribution compared to the lognormal, Weibull, and other extreme Type-I distributions for modeling maximum annual streamflow data was shown.  \cite{suto15} considered the log-logistic distribution to model water demand data. 
Other interesting studies in hydrological analysis that apply the log-logistic distribution include the works of \cite{cun89, gucu92, haho93, rostsi02}, and references therein. For Economics applications, we refer to \cite{klko03} and the original paper of Fisk \cite{fis61}, where it is shown that the log-logistic distribution is an appealing choice for modeling the distribution of wealth or income. Further, the log-log distribution has been widely applied in Survival Analysis to model for events in which the rate initially increases and then decreases, see e.g. the paper of Collet \cite{col03} where it is applied to cancer mortality after treatment. Finally, it is possible to develop control charts and the corresponding control limits based on log-logistic distribution the same way as Shewhart-charts, see \cite{kavasr06}.

Different estimation methods for the log-logistic distribution parameters, $\alpha$ and $\beta,$ have been considered in the literature. For example,  \cite{bamapu87} developed BLUE (Best linear unbiased estimators) for both parameters. \cite{kasr02} studied the maximum likelihood estimator (MLE) of the scale parameter $\alpha$ for known shape parameter $\beta,$ and ater, \cite{che06} studied the estimation of the shape parameter for known scale.
 
From maximum likelihood theory, it can be seen that the MLEs of $\alpha$ and $\beta $ are given by the solution of the following system of the two equations
\begin{equation*}
\left \{
\begin{array}{l}
-\dfrac{n\beta }{\alpha }+\dfrac{2\beta }{\alpha }\sum_{i=1}^{n}\log x_{i}-2\sum_{i=1}^{n}\log \left( 1+\left( \frac{x_{i}}{\alpha }\right) ^{\beta }\right) =0 \\
-\dfrac{n}{\beta }-n\log \alpha +\sum_{i=1}^{n}\log x_{i}-2\sum_{i=1}^{n}\left( \frac{x_{i}}{\alpha }\right) ^{\beta }\log \frac{x_{i}}{\alpha }\left( 1+\left( \frac{x_{i}}{\alpha }\right)^{\beta }\right) ^{-1}=0.
\end{array}
\right.
\end{equation*}

Because no closed-form solution of the MLEs can be found, numerical methods such as the Newton-Raphson or quasi-Newton algorithms are needed to obtain approximated MLEs. Some available optimization methods to obtain the MLEs can be found e.g. in \cite{prflteve86, lan99}.

In terms of efficiency, the asymptotic properties of the MLE are well established, as it is a Best Asymptotically Normal (BAN) estimator.
However, it is also well-known that the MLE has two important drawbacks: It might be a biased estimator for small or moderate sample sizes, and it might be heavily affected by data contamination. To correct the bias on the estimation, several approaches have appeared in the literature, see e.g. \cite{giles13, giles16} 
\cite{reath2018, wang2017}. 
This method reduces the bias of MLEs to the second order of magnitude by subtracting the estimated bias from the MLEs. On the other hand, regarding the lack of robustness, the use of minimum distance estimators based on the density power divergence (DPD) for the log-logistic distribution were proposed as a robust alternative to the MLE in \cite{mawapa23} and \cite{fejamipa24}. The minimum density power divergence estimator (MDPDE) family generalizes the MLE by a tuning parameter $\tau$ controlling the trade-off between efficiency and robustness on the estimation.

Since the log-logistic distribution is widely used in different applied fields, defining statistical test statistics allow researchers to determine assess specific hypotheses related to lifetime characteristics, failure rates, and risk factors. In this paper, robust test statistics based on minimum distance estimator for the log-logistic parameters are developed.
Particularly,
the Rao test (also known as the Score test) and the Wald test are two fundamental statistical tests based on the MLE used for hypothesis testing in parametric models. They are flexible and allow testing a wide range of hypotheses related to model parameters. However, there is no definitive choice between the two, as both have proven to perform better in different applications depending on the context. 
In this work  we aim to develop Wald and Rao-type test statistics based on the MDPDE, thereby extending the classical Wald and Rao tests to a broader family of statistics that remain resilient in the presence of outliers, and effectively addressing the robustness issue of the MLE. 
%

The rest of the paper is organized as follows. In Section \ref{sec:2} some key properties of the MDPDEs are presented. Section 3 deals with Wald-type tests, both for simple and composite null hypothesis. Section 4 presents the corresponding Rao-type tests for the same hypothesis tests. In Section 5 an extensive simulation study is carried out, where the behavior of test statistics based on MLE is compared to the robust statistics based on MDPDEs in terms of efficiency and robustness. 

\section{The Minimum Density Power Divergence Estimator \label{sec:2}}

Let $\mathcal{G}$ denote the set of all distributions having densities with respect to a dominating measure (generally the Lebesgue measure or the counting measure). Divergence measures are functionals aiming to measure similarity or discrepancy between two probability distributions on $\mathcal{G}.$ Formally, a divergence measure $d$  on two probability distributions, $g$ and $f,$ is a functional satisfying two properties: $d(g,f)\geq 0$ for all $g, f\in \mathcal{G}$, and $d_{\tau }(g,f)$ equals zero if and only if the two densities $g$ and $f$ are identically equal almost surely. Many divergence measures have been proposed in the literature (\cite{par06}). 
In this paper, we will consider the family of {\it density power divergences}, originally defined in \cite{bahahjjo98}.

Given any two densities $g$ and $f$ in $\mathcal{G}$, the DPD between them is defined as the function of a non-negative tuning parameter $\tau$ as follows:
\begin{equation}
d_{\tau }(g,f)=\int \left\{ f^{1+\tau }(x)-\left( 1+\frac{1}{\tau }\right) f^{\tau }(x)g(x)+\frac{1}{\tau }g^{1+\tau }(x)\right\} dx.  \label{W1.1411}
\end{equation}
The DPD can be extended at $\tau =0$ from the general case by taking the continuous limit as $\tau \rightarrow 0$, and in this case,
\begin{equation*}
d_{0}(g,f)=\lim_{\tau \rightarrow 0}d_{\tau }(g,f)=\int g(x)\log f(x)dx
\end{equation*}
coincides with the classical Kullback-Leibler divergence, yielding the maximum likelihood from an information theory approach. For more details about Kullback-Leibler divergence and it properties, see \cite{par06}. 
On the other hand, for $\tau =1,$ the square of the standard $L_{2}$ distance between $g$ and $f$ is obtained. More details about properties of DPD can be found in \cite{bahahjjo98}.

Now, consider a set of data coming from an unknown distribution $G$ and corresponding density function  $g.$ 
To model the underlying distribution, a parametric family 
$\mathcal{F}_{\boldsymbol{\theta}} = \{f_{\boldsymbol{\theta}} | \boldsymbol{\theta \in \Theta }\subset \mathbb{R}^{k}\},$  is assumed. Therefore, we are interested in estimating the distribution parameter $\boldsymbol{\theta}$ such that the assumed parametric distribution is as closely as possible to the true distribution of the data.
To measure such closeness, a divergence measure will be used. 

For a given divergence $d,$ the {\it minimum divergence functional} at $G$, denoted by $\boldsymbol{T}_{d}(G)$, is defined as
\begin{equation*}
d(g,f_{\boldsymbol{T}_{\beta }(G)})=\min_{\theta \in  \Theta }d(g,f_{\boldsymbol{\theta }}).
\end{equation*}

Importantly, the  statistical functional $\boldsymbol{T}_{d}(G)$ based on Kullback-Leibler divergence between $g$ and $f_{\boldsymbol{\theta }}$ is equivalent to the  maximum log-likelihood functional.
But of course, we can consider other different divergence measures providing functionals with certain desirable properties. In this paper, we adopt the DPD approach for the sake of robustness. 

The {\it minimum DPD functional} at $G$, denoted by $\boldsymbol{T}_{\tau}(G)$, is defined as
\begin{eqnarray}
d_{\tau }(g,f_{\boldsymbol{T}_{\beta }(G)}) 
  & = &  \min_{\theta \in  \Theta } \int \left\{ f_{\boldsymbol{\theta}}^{1+\tau }(x)-\left( 1+\frac{1}{\tau }\right) f_{\boldsymbol{\theta}}^{\tau }(x)g(x) \right\} dx. \label{W1.1412}
\end{eqnarray}
We have removed the last term in Eq. (\ref{W1.1411}) as it does not depend on the parameter $\boldsymbol{\theta}.$

In practice, the underlying cdf $G$ and pdf $g$ of the population are unknown, and so they need to be estimated from a sample. A natural choice for estimating $G$ is to consider the empirical cdf, $G_n,$ associated to the random sample $X_{1},\ldots ,X_{n}.$ 
However, there is no straightforward  method for empirically estimating the pdf of the distribution. Kernel-based estimations derived from the sample may serve as a viable approach, but they will depend on the kernel choice.
Conveniently, for the minimum DPD functional it is sufficient to estimate the underlying cdf $G$, as the pdf appears only as a linear term in the equation Eq. (\ref{W1.1412}), and so can be rewritten as 
\begin{eqnarray}
	d_{\tau }(g,f_{\boldsymbol{T}_{\beta }(G)}) 
	& = &  \min_{\theta \in  \Theta } \int \left\{ f_{\boldsymbol{\theta}}^{1+\tau }(x)-\left( 1+\frac{1}{\tau }\right) f_{\boldsymbol{\theta}}^{\tau }(x)dG(x) \right\} dx. 
\end{eqnarray}
 Therefore,  the {\it minimum DPD estimator} (MDPDE) of $\boldsymbol{\theta}$ is defined by
\begin{equation}
\widehat{\boldsymbol{\theta }}_{\tau }=\boldsymbol{T}_{\tau }(G_{n}) = \arg \min_{\theta \in  \Theta } \left\{ \int  f_{\boldsymbol{\theta}}^{1+\tau }(x) dx-\left( 1+\frac{1}{\tau }\right) \sum_{i=1}^n  f_{\boldsymbol{\theta}}^{\tau }(X_i) \right\}.  \label{W1.1413}
\end{equation}

Now, minimizing the estimated is equivalent to maximize the surrogate function
\begin{equation}
H_{n,\tau }(\boldsymbol{\theta )=}\left( 1+\frac{1}{\tau }\right) \frac{1}{n}\sum_{i=1}^{n}f_{\boldsymbol{\theta }}^{\tau }(X_{i})-\int f_{\boldsymbol{\theta }}^{1+\tau }(x)dx-\frac{1}{\tau },  \label{W1.1414}
\end{equation}
because $\frac{1}{\tau}$ does not depend on the parameter$\boldsymbol{\theta.}$ 
Therefore, we can equivalently define the MDPDE of $\boldsymbol{\theta}$, $\widehat{\boldsymbol{\theta }}_{\tau}, $ as 
\begin{equation*}
	\widehat{\boldsymbol{\theta }}_{\tau }=\arg \max_{\boldsymbol{\theta }\in \Theta }H_{n,\tau }(\boldsymbol{\theta }).
\end{equation*}

From this equivalent definition, it becomes evident that
\begin{equation*}
H_{n,0}(\boldsymbol{\theta })=\lim_{\tau \rightarrow 0}H_{n,\beta }(\boldsymbol{\theta })=\frac{1}{n}\sum_{i=1}^{n}\log f_{\boldsymbol{\theta }}(X_{i}),
\end{equation*}
and so the MDPDE at $\tau=0$ is given by
\begin{equation}
	\widehat{\boldsymbol{\theta}}_0=\arg \max_{\boldsymbol{\theta }\in \Theta }H_{n,0}(\boldsymbol{\theta })=\arg \max_{\boldsymbol{\theta }\in \Theta }\frac{1}{n}\sum_{i=1}^{n}\log f_{\boldsymbol{\theta }}(X_{i}), \label{W1.1415}
\end{equation}%
coincides with the MLE.
Remark that the objective function $H_{n,\tau }(\boldsymbol{\theta })$ depends on the DPD tuning parameter $\tau$ which will controls the trade-off between robustness and efficiency; the larger $\tau$ is, the most robust but less efficient the resulting estimator will, and conversely, the smaller $\tau$, the more efficient but less robust. That is, the robustness of the estimator will improve when $\tau$ increases, but efficiency will improve with $\tau.$  Thus, the most efficient (and less robust) estimator of the MDPDE family is the MLE.

 \cite{fejamipa24} obtained the explicit expression of Eq. (\ref{W1.1414})
under the log-logistic distribution with parameters $\boldsymbol{\theta}= (\alpha , \beta),$ given as follows
\begin{eqnarray*}
H_{n,\tau }(\alpha ,\beta \boldsymbol{)} &\boldsymbol{=}&\left( 1+\frac{1}{\tau }\right) \frac{1}{n}\sum_{i=1}^{n}f_{\boldsymbol{\theta }}^{\tau }(X_{i})-\int_{0}^{\infty }f_{\boldsymbol{\theta }}^{1+\tau }(x)dx-\frac{1}{\tau } \\
&=&\frac{1}{n}\sum_{i=1}^{n}\left\{ \left( 1+\frac{1}{\tau }\right) \frac{\beta ^{\tau }\alpha ^{\tau \beta }X_{i}^{\tau \left( \beta -1\right) }}{\left( X_{i}^{\beta }+\alpha ^{\beta }\right) ^{2\tau }}-\left( \frac{\beta }{\alpha }\right) ^{\tau }B\left( \frac{\beta \tau +\tau +\beta }{\beta },\frac{\beta \tau -\tau +\beta }{\beta }\right) -\frac{1}{\tau }\right\}
\end{eqnarray*}
for $\tau >0$ and
\begin{equation*}
H_{n,0}(\beta ,\alpha \boldsymbol{)=}\frac{1}{n}\sum_{i=1}^{n}\left\{ \log \beta +\beta \log \frac{X_{i}}{\alpha }-2\log \left( 1+\left( \frac{X_{i}}{\alpha }\right) ^{\beta }\right) +C\right\} ,
\end{equation*}
for $\tau =0,$ where $C$ is a constant that does not depend on the parameters $\beta $ and $\alpha,$ and $B(\cdot, \cdot)$ the beta function,
\begin{equation*}
	B\left( a,b\right) =\int_{0}^{1}x^{a-1}\left( 1-x\right) ^{b-1}dx.
\end{equation*}
 In the following, the MDPDE of $\alpha $ and $\beta $ will denoted by $\widehat{\alpha}_{\tau }$ and $\widehat{\beta}_{\tau },$ respectively. 

Let us now discuss some asymptotic properties of MDPDE. \cite{bahahjjo98} proved that, in general statistical models under mild regularity conditions, the MDPDE verifies that
\begin{equation*}
\sqrt{n}\left( \widehat{\boldsymbol{\theta }}_{\tau }-\boldsymbol{\theta }_0 \right)
\underset{n\longrightarrow \infty }{\overset{\mathcal{L}}{\longrightarrow }}
\mathcal{N}\left( 0, \boldsymbol{J}_{\tau }^{-1}(\boldsymbol{\theta}_0) \boldsymbol{K}_{\tau }(\boldsymbol{\theta}_0) \boldsymbol{J}_{\tau }^{-1}(\boldsymbol{\theta}_0)\right) ,
\end{equation*}
with
\begin{equation}\label{J}
\begin{aligned}
	\boldsymbol{J}_{\tau }(\boldsymbol{\theta}) &= \int_{0}^{\infty }\left( \frac{\partial \log f_{\boldsymbol{\theta}}(x)}{\partial \boldsymbol{\theta}}\right) ^{2} f_{\boldsymbol{\theta}}(x)^{\tau +1}dx\\
	\boldsymbol{K}_{\tau }(\boldsymbol{\theta}) &= \boldsymbol{J}_{2\tau }(\boldsymbol{\theta})-\boldsymbol{\xi }_{\tau }\left( \boldsymbol{\theta}\right) \boldsymbol{\xi }_{\tau }\left( \boldsymbol{\theta}\right)^T,
\end{aligned}
\end{equation}
and
\begin{equation*}
\boldsymbol{\xi}_{\tau }\left( \boldsymbol{\theta} \right) = \int_{0}^{\infty }\frac{\partial \log f_{\boldsymbol{\theta}}(x)}{\partial \boldsymbol{\theta}} f_{\boldsymbol{\theta}}(x)^{\tau +1}dx.
\end{equation*}

\cite{fejamipa24} obtained the explicit expressions of $\boldsymbol{J}_{\tau }(\boldsymbol{\theta}_0), \boldsymbol{K}_{\tau }(\boldsymbol{\theta})$ and $\boldsymbol{\xi }_{\tau }\left( \boldsymbol{\theta}\right)$ under the log-logistic distribution, given as follows:

\begin{equation}
	\begin{aligned}
		\boldsymbol{J}_{\tau }(\alpha ,\beta )&= \left( \begin{array}{cc} J_{\tau }^{11}\left( \alpha ,\beta \right) & J_{\tau }^{12}\left( \alpha ,\beta \right) \\
			J_{\tau }^{12}\left( \alpha ,\beta \right) & J_{\tau }^{22}\left( \alpha , \beta \right) \end{array}\right)  \\
		\boldsymbol{K}_{\tau }(\alpha ,\beta ) &=  \left( \begin{array}{cc} J_{2\tau }^{11}\left( \alpha ,\beta \right) & J_{2\tau }^{12}\left( \alpha ,\beta \right) \\
		J_{2\tau }^{12}\left( \alpha ,\beta \right) & J_{2\tau }^{22}\left( \alpha , \beta \right) \end{array}\right)
		- \left( \begin{array}{cc} \xi _{\tau }\left( \alpha \right)^2 & \xi _{\tau }\left( \alpha \right)\xi _{\tau }\left( \beta \right) \\
		\xi _{\tau }\left( \alpha \right)\xi _{\tau }\left( \beta \right) & \xi _{\tau }\left( \beta \right)^2 \end{array}\right)
	\end{aligned}
	\label{W1.1421}
\end{equation}
where
\begin{equation*}
	\begin{aligned}
	J_{\tau }^{11} \left( \alpha ,\beta \right) &= \left( \frac{\beta }{\alpha } \right)^{\tau +2}
	B\left( \frac{\beta \tau +\tau +\beta }{\beta }, \frac{\tau \beta -\tau +\beta }{\beta }\right)
	\left[ 2\frac{\left( \beta \tau +\tau +\beta \right) \left( -\tau \beta -\beta +\tau \right) }{\beta ^{2}\left( \tau +1\right) \left( 2\tau +3\right) } +1 \right], \\
	J_{\tau }^{22}\left( \alpha ,\beta \right) &= N_{1}+N_{2}+N_{3}+N_{4}+N_{5}+N_{6},\\
	J_{\tau }^{12}\left( \alpha ,\beta \right) &= J_{\tau }^{21}\left( \alpha ,\beta \right) = B_{1}+B_{2}+B_{3}+B_{4}+B_{5}+B_{6},\\
	\end{aligned}
\end{equation*}
with
\begin{enumerate}
	\item[i)] $N_{1}=\dfrac{\beta ^{\tau -2}}{\alpha ^{\tau }}B\left( \frac{\tau \beta +\tau +\beta }{\beta },\frac{\tau \beta -\tau +\beta }{\beta }\right) .$
	
	\item[ii)] $N_{2}=2\frac{\beta ^{\tau -2}}{\alpha ^{\tau }}B(\frac{\tau \beta +\tau +\beta }{\beta },\frac{\tau \beta -\tau +\beta }{\beta })
	\left\{ \Psi \left( \frac{\tau \beta -\tau +\beta }{\beta }\right) -\Psi \left( \frac{\tau \beta +\tau +\beta }{\beta }\right) \right \} .$
	
	\item[iii)] $N_{3}=-4\frac{\beta ^{\tau -2}}{\alpha ^{\tau }}B\left( \frac{\tau \beta +\beta +\tau }{\beta },\frac{\tau \beta +2\beta -\tau }{\beta }\right) \left \{ \Psi \left( \frac{\tau \beta +2\beta -\tau }{\beta }\right) -\Psi \left( \frac{\tau \beta +\beta +\tau }{\beta }\right) \right \} .$
	
	\item[iv)] $N_{4}=\frac{\beta ^{\tau -2}}{\alpha ^{\tau }}B\left( \frac{\tau \beta +\beta +\tau }{\beta },\frac{\tau \beta +\beta -\tau }{\beta }\right)
	\left\{ \left[ \Psi \left( \frac{\tau \beta +\beta -\tau }{\beta }\right) -\Psi \left( \frac{\tau \beta +\beta +\tau }{\beta }\right) \right]^{2}+
	\Psi^{\prime }\left( \frac{\tau \beta +\beta -\tau }{\beta }\right) +\Psi^{\prime }\left( \frac{\tau \beta +\beta +\tau }{\beta }\right) \right\} .$
	
	\item[v)] $N_{5}=-4\frac{\beta ^{\tau -2}}{\alpha ^{\tau }}B\left( \frac{\tau \beta +\beta +\tau }{\beta },\frac{\tau \beta +2\beta -\tau }{\beta }\right) \left\{ \left( \Psi \left( \frac{\tau \beta +2\beta -\tau }{\beta }\right) -\Psi \left( \frac{\tau \beta +\beta +\tau }{\beta }\right) \right)^{2} \right. $
	
	$\left. +\left( \Psi ^{\prime }\left( \frac{\tau \beta +2\beta -\tau }{\beta }\right) +\Psi ^{\prime }\left( \frac{\tau \beta +\beta +\tau }{\beta }\right) \right) \right\} .$
	
	\item[vi)] $N_{6}=4\frac{\beta ^{\tau -2}}{\alpha ^{\tau }}B\left( \frac{\tau \beta +\beta +\tau }{\beta }, \frac{\tau \beta +3\beta -\tau }{\beta }\right) \left\{ \left[ \Psi \left( \frac{\tau \beta +3\beta -\tau }{\beta }\right) -\Psi \left( \frac{\tau \beta +\beta +\tau }{\beta }\right) \right]^{2} \right. $
	
	$\left. +\left( \Psi ^{\prime }\left( \frac{\tau \beta +3\beta -\tau }{\beta }\right) +\Psi ^{\prime }\left( \frac{\tau \beta +\beta +\tau }{\beta }\right) \right) \right\} .$
\end{enumerate}
and
\begin{enumerate}
	\item[i)] $B_{1}=\frac{\beta^{\tau }}{\alpha^{\tau +1}}B\left( \frac{\tau \beta +\beta +\tau }{\beta },\frac{\tau \beta +\beta -\tau }{\beta }\right) .$
	
	\item[ii)] $B_{2}=\frac{\beta^{\tau }}{\alpha^{\tau +1}}B\left( \frac{\tau \beta +\beta +\tau }{\beta },\frac{\tau \beta +\beta -\tau }{\beta }\right)
	\left \{ \Psi \left( \frac{\tau \beta +\beta -\tau }{\beta }\right) - \Psi \left( \frac{\tau \beta +\beta +\tau }{\beta }\right) \right \} .$
	
	\item[iii)] $B_{3}=-2\frac{\beta^{\tau }}{\alpha^{\tau +1}}B\left( \frac{\tau \beta +\beta +\tau }{\beta },\frac{\tau \beta +2\beta -\tau }{\beta }\right) \left \{ \Psi \left( \frac{\tau \beta +2\beta -\tau }{\beta }\right) -\Psi \left( \frac{\tau \beta +\beta +\tau }{\beta }\right) \right \} .$
	
	\item[iv)] $B_{4}=-2\frac{\beta^{\tau }}{\alpha^{\tau +1}}B\left( \frac{\tau \beta +2\beta +\tau }{\beta },\frac{\tau \beta +\beta -\tau }{\beta }\right) .$
	
	\item[v)] $B_{5}=-2\frac{\beta ^{\tau }}{\alpha ^{\tau +1}}B\left( \frac{\tau \beta +2\beta +\tau }{\beta },\frac{\tau \beta +\beta -\tau }{\beta }\right) \left\{ \Psi \left( \frac{\tau \beta +\beta -\tau }{\beta }\right) -\Psi \left( \frac{\tau \beta +2\beta +\tau }{\beta }\right) \right\} .$
	
	\item[vi)] $B_{6}=4\frac{\beta ^{\tau }}{\alpha ^{\tau +1}}B\left( \frac{\tau \beta +2\beta +\tau }{\beta },\frac{\tau \beta +2\beta -\tau }{\beta }\right) \left\{ \Psi \left( \frac{\tau \beta +2\beta -\tau }{\beta }\right) - \Psi \left( \frac{\tau \beta +2\beta +\tau }{\beta }\right) \right\} .$
\end{enumerate}
Here $\Psi \left( x\right) $ denotes the digamma function defined as the logarithmic derivative of the gamma function and by $\Psi' \left( x\right) $ its derivative, the trigamma function.

On the other hand, under a log-logistic distribution, the vector $\boldsymbol{\xi }_{\tau }\left( \alpha ,\beta \right)$ is given by
\begin{equation}
	\boldsymbol{\xi }_{\tau }\left( \alpha ,\beta \right) =\left( \xi _{\tau }\left( \alpha \right) ,\xi _{\tau }\left( \beta \right) \right) \label{xifunc}
\end{equation}
with
\begin{enumerate}
	\item[i)] $\xi _{\tau }\left( \alpha \right) = \left( \frac{\beta }{\alpha }\right)^{\tau +1}
	B\left( \frac{\tau \beta +\beta +\tau }{\beta },\frac{\tau \beta +\beta -\tau }{\beta }\right) \left( \frac{-\tau }{\beta +\tau \beta }\right) .$
	
	\item[ii)] $\xi _{\tau }\left( \beta \right) = \frac{\beta ^{\tau -1}}{\alpha ^{\tau }}
	B\left( \frac{\tau \beta +\beta +\tau }{\beta },\frac{\tau \beta +\beta -\tau }{\beta }\right)
	\left\{ 1+\Psi \left( \frac{\tau \beta +\beta -\tau }{\beta }\right) \right. $
	
	$\left. +\frac{\left( \tau \beta +\beta -\tau \right) }{\beta \left( 2\tau +2\right) } \Psi \left( \frac{\tau \beta +2\beta -\tau }{\beta }\right) +
	\frac{\left( 3\tau \beta +3\beta -\tau \right) }{2\tau \beta +2\beta }\Psi \left( \frac{\tau \beta +\beta +\tau }{\beta }\right) \right\} .$
\end{enumerate}

Hence, substituting the log-logistic density on the previous expressions, \cite{fejamipa24} obtained the asymptotic distribution of the MDPDE as follows:
\begin{equation*}
\sqrt{n}\left( \widehat{\alpha }_{\tau }- \alpha_0 \right)
\underset{n\longrightarrow \infty }{\overset{\mathcal{L}}{\longrightarrow }}
\mathcal{N}\left( 0, J_{\tau }^{11}(\alpha_0)^{-1} K^{11}_{\tau }(\alpha_0) J_{\tau }^{11}(\alpha_0)^{-1}\right) . \label{W1.1416}
\end{equation*}

\begin{equation*}
\sqrt{n}\left( \widehat{\beta }_{\tau }- \beta_0 \right)
\underset{n\longrightarrow \infty }{\overset{\mathcal{L}}{\longrightarrow }}
\mathcal{N}\left( 0, J_{\tau }^{22}(\beta_0)^{-1} K^{22}_{\tau }(\beta_0) J_{\tau }^{22}(\beta_0)^{-1}\right) . \label{W1.1417}
\end{equation*}



\section{Wald-type tests based on MDPDE}

In this Section we introduce Wald-type tests based on MDPDEs for testing simple and composite null hypotheses.

\subsection{Simple null hypothesis}

Let us first consider the simple null hypothesis given by
\begin{equation}
H_{0}:\boldsymbol{\theta } = \boldsymbol{\theta }_0 \text{ versus }H_{1}: \boldsymbol{\theta } \neq \boldsymbol{\theta }_0. \label{W1.5}
\end{equation}
where $\boldsymbol{\theta} = (\alpha,\beta) \in \mathbb{R}^2.$
Given  a random simple of size $n$  from a log-logistic distribution, $X_{1},....X_{n},$
the Wald test statistic under the null hypothesis (\ref{W1.5}) is given by
$$ W_n(\widehat{\boldsymbol{\theta }}_{0}) = n (\widehat{\boldsymbol{\theta}}_{0 } - \boldsymbol{\theta }_0)^T \boldsymbol{I}^{-1}(\boldsymbol{\theta }_0) (\widehat{\boldsymbol{\theta }}_{0} - \boldsymbol{\theta }_0),$$ where $\widehat{\boldsymbol{\theta }}_{0} $ is the MLE of the log-logistic model parameters $\boldsymbol{\theta} = (\alpha,\beta)$ and $\boldsymbol{I}(\boldsymbol{\theta }_0)$ is the Fisher information matrix  at $\boldsymbol{\theta }_0.$ 
If the null hypothesis holds, the Wald test statistic converges asymptotically to a $\chi^2_2$ distribution.
Hence, at a significance level of $\alpha,$ 
the null hypothesis should be rejected if the test statistic exceeds the $100(1-\alpha)$ quantile of the chi-square distribution with $2$ degrees of freedom $\chi^2_{2, \alpha }.$ That is, the reject region of the Wald-test is given as follows
$$\text{RC} = \{(x_1,...,x_n) : W_n(\widehat{\boldsymbol{\theta}}_0) > \chi^2_{2, \alpha} \}.$$

Based on the classical Wald test, \cite{bamamapa16} generalized the test using the MDPDEs. The robustness properties of the resulting Wald-type test statistics were therein established.  

\begin{definition}
The robust Wald-type test statistic based on the MDPDE with tuning parameter $\tau \geq 0,$ $\widehat{\boldsymbol{\theta }}_{\tau },$ for testing (\ref{W1.5}) is given by
\begin{equation}
W_{n}\left( \widehat{\boldsymbol{\theta }}_{\tau }\right) = n \left( \widehat{\boldsymbol{\theta }}_{\tau } - \boldsymbol{\theta }_0\right)^{T} \boldsymbol{\Sigma }(\boldsymbol{\theta }_0)^{-1} \left( \widehat{\boldsymbol{\theta }}_{\tau }- \boldsymbol{\theta }_0\right) , \label{W1.6}
\end{equation}
with
\begin{equation*}
\boldsymbol{\Sigma }(\boldsymbol{\theta }_{0})= \boldsymbol{J}_{\tau }^{-1}(\boldsymbol{\theta }_{0}) \boldsymbol{K}_{\tau }(\boldsymbol{\theta }_{0}) \boldsymbol{J}_{\tau }^{-1}(\boldsymbol{\theta }_{0}).
\end{equation*}
\end{definition}

As expected, if we consider the limiting case $\tau =0,$ the MDPDE $\widehat{\boldsymbol{\theta }}_{\tau=0 }$ coincides with the MLE of $\boldsymbol{\theta}$ and $\boldsymbol{J}_{\tau=0}^{-1}(\boldsymbol{\theta }_{0}), \boldsymbol{K}_{\tau =0}(\boldsymbol{\theta }_{0}), \boldsymbol{J}_{\tau =0}^{-1}(\boldsymbol{\theta }_{0})$ with the inverse of the Fisher information matrix $I^{-1}(\widehat{\boldsymbol{\theta }}_{0}).$ Therefore, the classical Wald test statistic is obtained at $\tau=0.$

As with the classical test, an asymptotic distribution is required to define an asymptotic rejection region for the test. The key idea behind this definition is to replace the estimator with another consistent and robust estimator while substituting the inverse of the Fisher information matrix, which represents the asymptotic variance, with the asymptotic variance matrix of the MDPDE. Therefore, a similar asymptotic behaviour to that of the MLE is expected for this generalization.

For the special case of the log-logistic distribution with two unknown parameters, $\boldsymbol{\theta} = (\alpha, \beta)$, the Wald-type test statistic based on the MDPDE $\left(\widehat{\alpha }_{\tau }, \widehat{\beta }_{\tau }\right)$ is given by
%
%
\begin{equation}
	W_{n}\left( \widehat{\alpha }_{\tau }, \widehat{\beta }_{\tau }\right) = n \left( \widehat{\alpha }_{\tau }- \alpha _{0}, \widehat{\beta }_{\tau }-\beta _{0}\right) \left( \boldsymbol{J}_{\tau }^{-1}(\alpha _{0},\beta _{0}) \boldsymbol{K}_{\tau }(\alpha_{0}, \beta _{0}) \boldsymbol{J}_{\tau }^{-1}(\alpha _{0}, \beta _{0})\right)^{-1} \left( \widehat{\alpha }_{\tau }-\alpha _{0}, \widehat{\beta }_{\tau }-\beta _{0}\right) ^{T},  \label{W1.7}
\end{equation}
where matrices $\boldsymbol{J}_\tau$ and $\boldsymbol{K}_\tau$ are given as in (\ref{W1.1421}).
Similarly, for known shape $\beta$ and known scale $\alpha$ we may define the simple null hypotheses
\begin{equation}
	H_{0}:\alpha =\alpha _{0}\text{ versus }H_{1}:\alpha \neq \alpha _{0},( \beta \text{ known})  \label{W1.9}
\end{equation}
and
\begin{equation}
	H_{0}:\beta =\beta _{0}\text{ versus }H_{1}:\beta \neq \beta _{0}\quad ( \alpha \text{ known}),  \label{W1.10}
\end{equation}
respectively, and the corresponding Wald-type test statistics are given by
\begin{equation*}
	W_{n}\left( \widehat{\alpha }_{\tau }\right) = n \left( \widehat{\alpha }_{\tau }-\alpha_{0}\right)^{2} \frac{J_{\tau }^{11}\left( \alpha_0 \right)^2}{K_{\tau }^{11}(\alpha_0 )}
\end{equation*}
and
\begin{equation*}
	W_{n}\left( \widehat{\beta }_{\tau }\right) = n \left( \widehat{\beta }_{\tau }-\beta_{0}\right)^{2} \frac{J_{\tau }^{22}\left( \beta_0 \right)^2}{K_{\tau }^{22}(\beta_0)},
\end{equation*}
respectively where $J_{\tau }^{ii}$ and $K_{\tau }^{ii}, i =1,2$ are 
	the diagonal entries of matrices $\boldsymbol{J}_\tau(\alpha, \beta)$ and $\boldsymbol{K}(\alpha, \beta)$ given in (\ref{W1.1421}) with fixed values of the known parameter, accordingly. 


\begin{theorem}
Under the null hypothesis $H_0$ given in (\ref{W1.5}), 
the asymptotic distribution of the Wald-type test statistics based on a log-logistic distribution with two unknown parameters, $W_{n}\left( \widehat{\boldsymbol{\theta }}_{\tau }\right)$  is a chi-square distribution with \textcolor{black}{2} degrees of freedom, and the the asymptotic distribution of the Wald-type test statistics based on a log-logistic distribution with just one unknown parameters, $W_{n}\left( \widehat{\alpha}_{\tau }\right)$ and $W_{n}\left( \widehat{\beta}_{\tau}\right),$ is a chi-square distribution with \textcolor{black}{1} degrees of freedom,
\end{theorem}

Based on the previous discussion, the null hypothesis given in (\ref{W1.5}) should be rejected at significance level $\alpha$ if
\begin{equation*}
	W_{n}\left( \widehat{\alpha }_{\tau },\widehat{\beta }_{\tau }\right) >\chi_{2,\alpha }^{2}
\end{equation*}
where $\chi _{r,\alpha }^{2}$ denotes the $100(1-\alpha )$ quantile point of chi-square distribution with r degrees of freedom. 
Similarly, the null hypotheses given in (\ref{W1.9}) and (\ref{W1.10}) should be rejected at significance level $\alpha$ if
\begin{equation*}
	W_{n}\left( \widehat{\alpha }_{\tau }\right) >\chi_{1,\alpha }^{2} \text{ and }
	W_{n}\left( \widehat{\beta }_{\tau }\right) >\chi_{1,\alpha }^{2}
\end{equation*}
respectively.

%

%

Importantly, if we denote by
$
\beta _{W_{n}\left( \alpha _{0},\beta _{0}\right) }(\alpha ^{\ast },\beta^{\ast })
$
the power function of the Wald-type tests under two unknown parameters, we have that
\begin{equation*}
\lim_{n\rightarrow \infty }W_{n}\left( \alpha _{0},\beta _{0}\right) =1,
\end{equation*}
i.e., the tests are consistent in the sense of Fraser \cite{fra57}.

\subsection{Composite null hypothesis}

Let us now consider composite null hypotheses. We denote by $\Theta = \mathbb{R}^{+}\times \mathbb{R}^{+}$ 
the parameter space associated to the log-logistic model, and
\begin{equation}
	\Theta _{0}=\left \{ \left( \alpha ,\beta \right) \in \Theta /\boldsymbol{m} \left( \alpha ,\beta \right) =\boldsymbol{0}_{r}\right \}  \label{W.Null}
\end{equation}
a restricted parameter defined through a differentiable function
\begin{equation}
	\boldsymbol{m}:\text{ }\mathbb{R}^{+}\times \mathbb{R}^{+}\rightarrow \mathbb{R}^{r}\text{ }(r\leq 2).  \label{W1.42}
\end{equation}
 Therefore, the Jacobian matrix of the function $\boldsymbol{m},$ here denoted by
\begin{equation}
	\boldsymbol{M}\left( \alpha ,\beta \right) =\frac{\partial \boldsymbol{m} \left( \alpha ,\beta \right) ^{T}}{\boldsymbol{\partial }\left( \alpha , \beta \right) }, \label{W1.43}
\end{equation}
exists and is continuous in $\left( \alpha ,\beta \right) $ with $rank\left( \boldsymbol{M}\left( \alpha ,\beta \right) \right) =r.$ 

Our interest is in testing composite null hypothesis of the form
\begin{equation}
H_{0}:\left( \alpha ,\beta \right) \in \Theta _{0}\text{ versus } H_{1}:\left( \alpha ,\beta \right) \notin \Theta _{0}.  \label{W1.5_2}
\end{equation}
Note that the simple null hypothesis given in (\ref{W1.5}), (\ref{W1.9}) and (\ref{W1.10}) could be alternatively defined using the linear functions
 $\boldsymbol{m}\left( \alpha ,\beta \right) =\left( \alpha -\alpha_{0}, \beta -\beta _{0}\right)^{T},$   $\boldsymbol{m}\left( \alpha ,\beta \right) =\beta -\beta_{0}$ and $\boldsymbol{m}\left( \alpha ,\beta \right) =\alpha -\alpha _{0},$ respectively.
%
%
 For composite null hypothesis, te classical Wald test statistic based on the MLE, $\boldsymbol{\widehat{\theta}}_{\tau=0},$ for this test is given by
$$ W_n(\widehat{\boldsymbol{\theta}}_{\tau=0}) = n \boldsymbol{m}(\widehat{\boldsymbol{\theta }}_{\tau=0})^T \boldsymbol{I}^{-1}(\widehat{\boldsymbol{\theta }}_{\tau=0}) \boldsymbol{m}(\widehat{\boldsymbol{\theta }}_{\tau=0}).$$

Hence, as before, the generalized Wald-type test statistic based on the MDPDE under composite null hypothesis is defined as follows.
\begin{definition}
The Wald-type test statistic based on the MDPDE with tuning paramter $\tau$ for testing (\ref{W1.5}) under a log-logistic model is given by
\begin{equation}
W_{n}\left( \widehat{\alpha }_{\tau }, \widehat{\beta }_{\tau }\right) =n \boldsymbol{m}\left( \widehat{\alpha }_{\tau },\widehat{\beta }_{\tau }\right)^{T} \boldsymbol{\Sigma }(\widehat{\alpha }_{\tau },\widehat{\beta }_{\tau })^{-1}\boldsymbol{m}\left( \widehat{\alpha }_{\tau },\widehat{\beta }_{\tau }\right) , \label{W1.6}
\end{equation}
with
\begin{equation*}
\boldsymbol{\Sigma }(\widehat{\alpha }_{\tau },\widehat{\beta }_{\tau })= \boldsymbol{M}^{T}(\widehat{\alpha }_{\tau },\widehat{\beta }_{\tau }) \boldsymbol{J}_{\tau }^{-1}(\widehat{\alpha }_{\tau },\widehat{\beta }_{\tau })\boldsymbol{K}_{\tau }(\widehat{\alpha }_{\tau },\widehat{\beta }_{\tau }) \boldsymbol{J}_{\tau }^{-1}(\widehat{\alpha }_{\tau },\widehat{\beta }_{\tau })\boldsymbol{M}(\widehat{\alpha }_{\tau },\widehat{\beta }_{\tau })
\end{equation*}
and the matrices $\boldsymbol{M}(\boldsymbol{\theta }),\boldsymbol{J}_{\tau }\left( \boldsymbol{\theta }\right) $ and $\boldsymbol{K}_{\tau
}\left( \boldsymbol{\theta }\right) $ were defined in (\ref{W1.43}), (\ref{W1.1421}) and (\ref{xifunc}), respectively and the function $\boldsymbol{m}$ in (\ref{W1.42}).
\end{definition}

\textcolor{black}{Importantly, note that the variance-covariance matrix is evaluated at the MDPDE, making this test equivalent to the Wald test defined in (\ref{W1.6}) under the null hypothesis. However, in practice, the two expressions are not identical. }
Again, setting $\tau =0,$ $\boldsymbol{\widehat{\theta}}_0 = (\widehat{\alpha }_{\tau=0 },\widehat{\beta }_{\tau=0})$ coincides with the MLE of $(\alpha ,\beta)$ and $\boldsymbol{J}_{\tau =0}^{-1}(\widehat{\alpha }, \widehat{\beta })\boldsymbol{K}_{\tau =0}(\widehat{\alpha },\widehat{\beta }) \boldsymbol{J}_{\tau =0}^{-1}(\widehat{\alpha },\widehat{\beta })$ with the Fisher information matrix. Thus, we get the classical Wald test statistic as a special case of the Wald-type test family.


\begin{theorem}
The asymptotic distribution of the Wald-type test statistics given in (\ref{W1.6}) is a chi-square distribution with $r$ degrees of freedom.
\end{theorem}
The proof is given in \cite{}.

Based on the previous Theorem, the Wald-type test based on the MDPDE with tuning parameter $\tau$, should be reject with significance level $\alpha $ if
\begin{equation*}
W_{n}\left( \widehat{\alpha }_{\tau },\widehat{\beta }_{\tau }\right) >\chi_{r,\alpha }^{2}.
\end{equation*}

And for the power function,
\begin{equation*}
\lim_{n\rightarrow \infty }W_{n}\left( \alpha _{0},\beta _{0}\right) =1.
\end{equation*}


\section{Rao-type tests}

Finally, let us address the generalization of the Rao test based on MDPDE under the log-logistic distribution. A key advantage of the test under a simple null hypothesis is that it does not require parameter estimation, relying only on the computation of the Fisher information matrix. In contrast, for a composite null hypothesis, a restricted MLE is necessary. In this section, we discuss both cases separately and derive explicit expressions.


\subsection{Simple null hypothesis}

Consider the simple null hypotheses defined in (\ref{W1.5}), (\ref{W1.9}) and (\ref{W1.10}) for the log-logistic model parameters.
The classical Rao-type test statistic for testing (\ref{W1.5}) is given by
$$ R(\boldsymbol{\theta}_0)= n U(\boldsymbol{\theta}_0) I^{-1}(\boldsymbol{\theta}_0) U(\boldsymbol{\theta_0}),$$
with $I(\boldsymbol{\theta}_0)$ the fisher information matrix and
\textcolor{black}{$$ U(\boldsymbol{\theta }_0) = {1\over n} \sum_{i=1}^n u(X_i, \boldsymbol{\theta }_0) = {1\over n} \sum_{i=1}^n {\partial log f_{\boldsymbol{\theta}}(x)\over \boldsymbol{\theta}}\bigg|_{\boldsymbol{\theta} = \boldsymbol{\theta}_0},$$}
the score function of the log-logistic model (first derivative of log-likelihood).
Similar expressions can be defined for the tests (\ref{W1.9}) and (\ref{W1.10}).

It can be proved that, if the null hypothesis holds, then $ R(\boldsymbol{\theta}_0)\overset{\mathcal{L}}{\underset{n\rightarrow \infty }{\longrightarrow }} \chi^2_{2},$ and so an asymptotic reject region at a significance level $\alpha$ is defined as 
$$\text{RC} = \{(x_1,...,x_n) : R(\boldsymbol{\theta}_0) > \chi^2_{2, \alpha} \}.$$

We next define the generalization of the Rao test by using the MDPDE. Similarly as for the classical test, it is necessary to define and compute the score function of the DPD for estimating $\alpha$ and $\beta$. These score define the estimating equations of the MDPDE, and characterizes the robust estimator as an M-estimator. 

\begin{proposition} \label{prop:score_alpha}
		Under a log-logistic model the score function of the DPD with fixed $\tau$ for the scale parameter $\alpha$ is given by
		\textcolor{black}{\begin{eqnarray*}
				u_{\tau }(x,\alpha )&=&\frac{\partial \log f_{\alpha ,\beta }(x)}{\partial \alpha }f_{\alpha ,\beta }(x)^{\tau }-\int \frac{\partial \log f_{\alpha , \beta }(x)}{\partial \alpha }f_{\alpha ,\beta }(x)^{\tau +1}dx \\
				&=&\frac{\beta \left(
					-\alpha ^{\beta }+X^{\beta }\right) \beta ^{\tau }\alpha ^{\beta \tau
					}X^{\left( \beta -1\right) \tau }}{\alpha \left( \alpha ^{\beta
					}+X^{\beta }\right) ^{2\tau +1}} \\
				&&+\frac{1}{1+\tau}\frac{\tau }{\alpha }\left( \frac{\beta }{\alpha }\right) ^{\tau }B\left( 
				\frac{\beta \tau +\tau +\beta }{\beta },\frac{\beta \tau -\tau +\beta }{%
					\beta }\right).
		\end{eqnarray*}}
\end{proposition}
\begin{proof}
		We have, 
	\[
	f_{\alpha ,\beta }(x)=\frac{\beta \alpha ^{\beta }x^{\beta -1}}{\left(
		x^{\beta }+\alpha ^{\beta }\right)^2 }=\frac{\left( \frac{x}{\alpha }\right)
		^{\beta -1}\frac{\beta }{\alpha }}{\left( 1+\left( \frac{x}{\alpha }\right)
		^{\beta }\right) ^{2}}. 
	\]%
	Therefore,%
	\[
	\log f_{\alpha ,\beta }(x)=\left( \beta -1\right) \log \frac{x}{\alpha }%
	+\log \frac{\beta }{\alpha }-2\log \left( 1+\left( \frac{x}{\alpha }\right)
	^{\beta }\right) 
	\]%
	and 
	\begin{eqnarray*}
		\frac{\partial \log f_{\alpha ,\beta }(x)}{\partial \alpha } &=&-\frac{%
			\left( \beta -1\right) }{\alpha }-\frac{1}{\alpha }-\frac{2\beta \left( 
			\frac{x}{\alpha }\right) ^{\beta -1}\left( -\frac{x}{\alpha ^{2}}\right) }{%
			1+\left( \frac{x}{\alpha }\right) ^{\beta }} \\
		&=&\frac{-\beta \alpha ^{\beta }+\beta x^{\beta }}{\alpha \left( \alpha
			^{\beta }+x^{\beta }\right) }.
	\end{eqnarray*}%
	Then 
	\begin{eqnarray*}
		\frac{\partial \log f_{\alpha ,\beta }(x)}{\partial \alpha }f_{\alpha ,\beta
		}(x)^{\tau } &=&\frac{-\beta \alpha ^{\beta }+\beta x^{\beta }}{\alpha
			\left( \alpha ^{\beta }+x^{\beta }\right) }\left( \frac{\beta ^{\tau }\alpha
			^{\beta \tau }X_{i}^{\left( \beta -1\right) \tau }}{\left( \alpha ^{\beta
			}+x^{\beta }\right) ^{2\tau }}\right) \\
		&=&\frac{\beta \left( -\alpha ^{\beta }+x^{\beta }\right) \beta ^{\tau
			}\alpha ^{\beta \tau }X_{i}^{\left( \beta -1\right) \tau }}{\alpha \left(
			\alpha ^{\beta }+x^{\beta }\right) ^{2\tau +1}}.
	\end{eqnarray*}%
	On the other taking into account, see \cite{fejamipa24},
	that, 
	\[
	\int f_{\alpha ,\beta }(x)^{\tau +1}dx=\left( \frac{\beta }{\alpha }\right)
	^{\tau }B\left( \frac{\beta \tau +\tau +\beta }{\beta },\frac{\beta \tau
		-\tau +\beta }{\beta }\right) 
	\]%
	we have, 
	\begin{eqnarray*}
		\int \frac{\partial \log f_{\alpha ,\beta }(x)}{\partial \alpha }f_{\alpha
			,\beta }(x)^{\tau}dx &=&\frac{1}{1+\tau}\frac{\partial }{\partial \alpha }\int f_{\alpha
			,\beta }(x)^{\tau +1}dx \\
		&=&-\frac{1}{1+\tau}\frac{\tau }{\alpha }\left( \frac{\beta }{\alpha }\right) ^{\tau
		}B\left( \frac{\beta \tau +\tau +\beta }{\beta },\frac{\beta \tau -\tau
			+\beta }{\beta }\right) .
	\end{eqnarray*}%
\end{proof}

Similarly, the score function for $\beta$ is given in the next result.
\begin{proposition} \label{prop:score_beta}
		Under a log-logistic model the score function of the DPD with fixed $\tau$ for the scale parameter $\beta$ is given by
	\begin{eqnarray*}
		u_{\tau }(x,\beta )&=&\frac{\partial \log f_{\alpha ,\beta }(x)}{\partial \beta }f_{\alpha ,\beta }(x)^{\tau }-\int \frac{\partial \log f_{\alpha , \beta }(x)}{\partial \beta }f_{\alpha ,\beta }(x)^{\tau +1}dx.\\
		%
		%
		&=&\left( \log \frac{X_{i}%
		}{\alpha }\left( \frac{\alpha ^{\beta }-X_{i}^{\beta }}{\alpha ^{\beta
			}+X_{i}^{\beta }}\right) +\frac{1}{\beta }\right) \frac{\beta ^{\tau }\alpha
			^{\beta \tau }X_{i}^{\left( \beta -1\right) \tau }}{\left( \alpha ^{\beta
			}+X_{i}^{\beta }\right) ^{2\tau }} \\
		&&-\frac{1}{1+\tau}\left\{ \frac{\tau }{\alpha }\left( \frac{\beta }{\alpha }\right) ^{\tau
			-1}B\left( \frac{\beta \tau +\tau +\beta }{\beta },\frac{\beta \tau -\tau
			+\beta }{\beta }\right) \right. \\
		&&\left. +\left( \frac{\beta }{\alpha }\right) ^{\tau }B\left( \frac{\beta
			\tau +\tau +\beta }{\beta },\frac{\beta \tau -\tau +\beta }{\beta }\right)
		\right. \\
		&&\left. \times \frac{\tau }{\beta ^{2}}\left( \psi \left( \frac{\beta \tau
			-\tau +\beta }{\beta }\right) -\psi \left( \frac{\beta \tau -\tau +\beta }{%
			\beta }\right) \right) \right\} .
	\end{eqnarray*}
\end{proposition}

\begin{proof}
	We can rewrite the partial of the density logartithm as follows
	\begin{eqnarray*}
		\frac{\partial \log f_{\alpha ,\beta }(x)}{\partial \beta } &=&\log \frac{x}{%
			\alpha }+\frac{1}{\beta }-2\frac{\left( \frac{x}{\alpha }\right) ^{\beta
			}\log \frac{x}{\alpha }}{1+\left( \frac{x}{\alpha }\right) ^{\beta }} \\
		&=&\log \frac{x}{\alpha }\left( \frac{\alpha ^{\beta }.x^{\beta }}{\alpha
			^{\beta }+x^{\beta }}\right) +\frac{1}{\beta }
	\end{eqnarray*}%
	and so
	\[
	\frac{\partial \log f_{\alpha ,\beta }(x)}{\partial \beta }f_{\alpha ,\beta
	}(x)^{\tau }=\left( \log \frac{x}{\alpha }\left( \frac{\alpha ^{\beta
		}-x^{\beta }}{\alpha ^{\beta }+x^{\beta }}\right) +\frac{1}{\beta }\right) 
	\frac{\beta ^{\tau }\alpha ^{\beta \tau }x^{\left( \beta -1\right) \tau }}{%
		\left( \alpha ^{\beta }+x^{\beta }\right) ^{2\tau }}. 
	\]%
	On the other hand, 
	\textcolor{black}{\begin{eqnarray*}
		\int \frac{\partial \log f_{\alpha ,\beta }(x)}{\partial \beta }f_{\alpha
			,\beta }(x)^{\tau}dx &=& \frac{1}{1+\tau} \frac{\partial }{\partial \beta }\int f_{\alpha
			,\beta }(x)^{\tau +1}dx \\
		&=&\frac{1}{1+\tau} \frac{\partial }{\partial \beta }\left[ \left( \frac{\beta }{\alpha }%
		\right) ^{\tau }B\left( \frac{\beta \tau +\tau +\beta }{\beta },\frac{\beta
			\tau -\tau +\beta }{\beta }\right) \right] \\
		&=&\frac{1}{1+\tau} \frac{\tau }{\alpha }\left( \frac{\beta }{\alpha }\right) ^{\tau
			-1}B\left( \frac{\beta \tau +\tau +\beta }{\beta },\frac{\beta \tau -\tau
			+\beta }{\beta }\right) \\
		&&+\frac{1}{1+\tau} \left( \frac{\beta }{\alpha }\right) ^{\tau }B\left( \frac{\beta \tau
			+\tau +\beta }{\beta },\frac{\beta \tau -\tau +\beta }{\beta }\right) \\
		&&\times \frac{\tau }{\beta ^{2}}\left( \psi \left( \frac{\beta \tau -\tau
			+\beta }{\beta }\right) -\psi \left( \frac{\beta \tau -\tau +\beta }{\beta }%
		\right) \right) .
	\end{eqnarray*}}%
\end{proof}

The Rao-type test statistics are defined through score functions of $\alpha$ and $\beta,$ $u_{\tau }(X,\alpha)$ and $u_{\tau }(X,\beta).$ Consequently, to understand the asymptotic behavior of the test we first establish the asymptotic distribution of the scores.

\begin{theorem} \label{thm:score}
Consider a random sample $X_1,...X_n,$ from a log-logistic model 
and define the score function $u_{\tau }(x, \alpha)$ and $u_{\tau }(x, \beta)$  for a fixed $\tau$ as in Proposition (\ref{prop:score_alpha}) and (\ref{prop:score_beta}), respectively. 
Then, it holds 
\begin{equation*}
	\begin{aligned}
		\sqrt{n}U_{\tau }(\alpha ) & \overset{\mathcal{L}}{\underset{n\rightarrow \infty }{\longrightarrow }} N\left( 0,K^{11}_{\tau }\left( \alpha, \beta \right) \right) .\\
		\sqrt{n}U_{\tau }(\beta )&\overset{\mathcal{L}}{\underset{n\rightarrow \infty }{\longrightarrow }} N\left( 0,K^{22}_{\tau }\left(\alpha, \beta \right) \right) .
	\end{aligned}
\end{equation*}
for $$U_{\tau }(\alpha ) = \frac{1}{n}\sum \limits_{i=1}^{n}u_{\tau }(X_{i},\alpha)
\text{ and }
U_{\tau }(\beta) = \frac{1}{n}\sum \limits_{i=1}^{n}u_{\tau }(X_{i},\beta),$$
and $K^{ii}_{\tau}, i=1,2$ are the diagonal entries of matrix $\boldsymbol{K}_\tau\left(\alpha, \beta \right)$ defined in (\ref{W1.1421}).
\end{theorem}

\begin{proof}
We only prove the result for the score of $\alpha,$ as the proof for $\beta$ follows similar steps.
We consider the random variable $u_{\tau }(X,\alpha)$ where $\tau$ is fixed.
Firs, it is clear that,
\begin{equation*}
E\left[ u_{\tau }(X,\alpha )\right] =0
\end{equation*}
and
\begin{eqnarray*}
V\left[ u_{\tau }(X,\alpha )\right]  &= & E\left[ \left( u_{\tau }(X,\alpha )\right) ^{2}\right] -E\left[ u_{\tau }(X,\alpha )\right] ^{2} \\
& = & E\left[ \left( \frac{\partial \log f_{\alpha ,\beta }(X)}{\partial \alpha }f_{\alpha ,\beta }(X)^{\tau }-\int \frac{\partial \log f_{\alpha ,\beta }(x)}{\partial \alpha }f_{\alpha ,\beta }(x)^{\tau +1}dx\right)^{2}\right]
\\
& = &E\left[ \left( \frac{\partial \log f_{\alpha ,\beta }(X)}{\partial \alpha }\right) ^{2}f_{\alpha ,\beta }(X)^{2\tau }\right] +\left( \int \frac{\partial \log f_{\alpha ,\beta }(x)}{\partial \alpha }f_{\alpha ,\beta }(x)^{\tau +1}dx\right) ^{2} \\
& & -2E\left[ \left( \frac{\partial \log f_{\alpha ,\beta }(X)}{\partial \alpha }\right) f_{\alpha ,\beta }(X)^{\tau }\int \frac{\partial \log f_{\alpha ,\beta }(x)}{\partial \alpha }f_{\alpha ,\beta }(x)^{\tau +1}dx\right] \\
& = & E\left[ \left( \frac{\partial \log f_{\alpha ,\beta }(X)}{\partial \alpha }\right) ^{2}f_{\alpha ,\beta }(X)^{2\tau }\right] -E\left[ \left( \frac{\partial \log f_{\alpha ,\beta }(X)}{\partial \alpha }\right) f_{\alpha
,\beta }(X)^{\tau }\right] \\
& = & K^{11}_{\tau }\left( \alpha, \beta \right) .
\end{eqnarray*}
Applying the Central Limit Theorem, the result holds.
\end{proof}

Let us first consider the simple null hypotheses given in (\ref{W1.9}) and (\ref{W1.10}) for known shape and scale parameters, $\beta^\ast$ and $\alpha^\ast,$ respectively. 

\begin{definition}
	The Rao-type test statistic based on the MDPDE with tuning parameter $\tau$ for testing (\ref{W1.9}) and (\ref{W1.10}) are given by
	\begin{equation}
		\begin{aligned}
			R_\tau(\alpha_0) &= nU_{\tau }\left( \alpha_0 \right) K^{11}_{\tau}\left( \alpha_0, \beta^\ast \right) ^{-1}\\
			R_\tau(\beta_0) &= nU_{\tau }\left( \beta_0 \right) K^{22}_{\tau }\left( \alpha^\ast, \beta_0 \right) ^{-1}
		\end{aligned}
	\end{equation}
	where $\alpha^\ast$ and $\beta^\ast$ are assumed to be known.
\end{definition}

By Theorem (\ref{thm:score}), we straightforward have the asymptotic convergence of both tests under the null hypothesis as follows
\begin{equation}
R_\tau(\alpha_0) = nU_{\tau }\left( \alpha_0 \right) K^{11}_{\tau}\left( \alpha_0, \beta^\ast \right) ^{-1}\overset{\mathcal{L}}{\underset{n\rightarrow \infty }{\longrightarrow }}\chi _{1}^{2}. \label{W1.11}
\end{equation}
and
\begin{equation}
R_\tau(\beta_0) =  nU_{\tau }\left( \beta \right) K^{22}_{\tau }\left(\alpha^\ast \beta \right) ^{-1}\overset{\mathcal{L}}{\underset{n\rightarrow \infty }{\longrightarrow }}\chi _{1}^{2} \label{W1.12}
\end{equation}

\textcolor{black}{The Rao-type test statistic for simple null hypothesis does not require any estimator of the model parameter, but an estimator of the asymptotic variance-covariance matrix of the test. Thus, the robustness of the test will be hirited from this matrix.} 

Finally, we consider the case of two unknown parameters and simple null hypothesis as given in (\ref{W1.5}). In that case, we consider the score vector as 
$
\mathbf{U}_{\tau }\left( \alpha ,\beta \right) =\left( U_{\tau }\left( \alpha \right) ,U_{\tau }\left( \beta \right) \right) ^{T}.
$
Then, it is immediate to see that
\begin{equation*}
\sqrt{n}\mathbf{U}_{\tau }\left( \alpha ,\beta \right) \overset{\mathcal{L}}{\underset{n\rightarrow \infty }{\longrightarrow }}N\left( 0,\boldsymbol{K}_{\tau }(\alpha ,\beta )\right)
\end{equation*}
with $\boldsymbol{K}_{\tau }(\alpha ,\beta )$ given in (\ref{W1.1421}).

\begin{definition}
		The Rao-type test statistic based on the MDPDE with tuning parameter $\tau$ for testing the simple null hypothesis (\ref{W1.5}) is given by
	\begin{equation}
			\boldsymbol{R}_\tau(\alpha_0,\beta_0) = n\mathbf{U}_{\tau }\left( \alpha_0 ,\beta_0 \right) \boldsymbol{K}_{\tau }(\alpha_0 ,\beta_0 )\mathbf{U}_{\tau }\left( \alpha_0 ,\beta_0 \right)^{T}.
	\end{equation}
	where $\mathbf{U}_{\tau }\left( \alpha ,\beta \right) =\left( U_{\tau }\left( \alpha \right) ,U_{\tau }\left( \beta \right) \right) ^{T}$ 
	and 
	 $\boldsymbol{K}_{\tau }(\alpha ,\beta )$ is given in (\ref{W1.1421}).
\end{definition}

Under the null hypothesis, it holds that
\begin{equation*}
	\boldsymbol{R}_\tau(\alpha_0,\beta_0) = \boldsymbol{R}_\tau(\alpha_0,\beta_0) = n\mathbf{U}_{\tau }\left( \alpha_0 ,\beta_0 \right) \boldsymbol{K}_{\tau }(\alpha_0 ,\beta_0)\mathbf{U}_{\tau }\left( \alpha ,\beta \right) ^{T}\overset{\mathcal{L}}{\underset{n\rightarrow \infty }{\longrightarrow }}\chi _{2}^{2}.
\end{equation*}

Hence, the Rao-type tests for testing the simple null hypotheses considered in (\ref{W1.5}), (\ref{W1.9}) and (\ref{W1.10}) reject the null hypothesis if

\begin{equation*}
nU_{\tau }\left( \alpha _{0}\right)^2 K_{\tau }^{11}\left( \alpha _{0}\right)^{-1}>\chi _{1,\alpha }^{2},
\end{equation*}

\begin{equation*}
nU_{\tau }\left( \beta _{0}\right)^2 K_{\tau }^{22}\left( \beta _{0}\right)^{-1}>\chi _{1,\alpha }^{2}
\end{equation*}
and

\begin{equation*}
n\mathbf{U}_{\tau }\left( \alpha _{0},\beta _{0}\right) \boldsymbol{K}_{\tau }(\alpha _{0},\beta _{0})\mathbf{U}_{\tau }\left( \alpha _{0}, \beta_{0}\right) ^{T}>\chi _{2}^{2},
\end{equation*}
respectively.

\subsection{Composite null hypothesis}

To introduce the Rao-type tests for evaluating the composite null hypothesis, it is essential to first define the restricted MDPDE (RMDPDE).
%
Restricted minimum distance estimators were introduced for the first time in \cite{papazo02} and since then, numerous studies have been published in this area, see for example \cite{bamamap18}.

For the log-logistic model with parameters $\boldsymbol{\theta }=\left( \alpha ,\beta \right) $ the RMDPDE, denoted by $\widetilde{\boldsymbol{\theta }}_{\tau } = \left( \widetilde{\alpha }_{\tau }, \widetilde{\beta }_{\tau }\right),$ is defined by
\begin{equation*}
\widetilde{\boldsymbol{\theta }}_{\tau }=\arg \max_{\boldsymbol{\theta }\in \boldsymbol{\Theta }_{0}}H_{n,\tau }(\boldsymbol{\theta }).
\end{equation*}
where $\Theta_0$ is a restricted parameter space of the form (\ref{W.Null}).

Some main asymptotic properties of the RMDPDE were established\cite{bamamap18} as follows:
\begin{enumerate}
\item If $\boldsymbol{\theta }_{0} \in \boldsymbol{\Theta }_{0},$ then the RMDPDE estimating equation has a consistent sequence of roots $\widetilde{\boldsymbol{\theta }}_{\tau }$ such that

\begin{equation*}
\widetilde{\boldsymbol{\theta }}_{\tau }\overset{\mathcal{P}}{\underset{n\rightarrow \infty }{\longrightarrow }}\boldsymbol{\theta }_{0}.
\end{equation*}

\item The asymptotic null distribution of $\widetilde{\boldsymbol{\theta }}_{\tau }$ is given by

\begin{equation*}
\sqrt{n}\left( \widetilde{\boldsymbol{\theta }}_{\tau }-\boldsymbol{\theta }_{0}\right) \overset{\mathcal{L}}{\underset{n\rightarrow \infty }{\longrightarrow }}\mathcal{N}\left( \boldsymbol{0},\boldsymbol{\Sigma }_{\tau }\left( \boldsymbol{\theta }_{0}\right) \right) ,
\end{equation*}
where

\begin{eqnarray*}
\boldsymbol{\Sigma }_{\tau }\left( \boldsymbol{\theta }_{0}\right) & = &\boldsymbol{P}_{\tau }\left( \boldsymbol{\theta }_{0}\right) \boldsymbol{K}_{\tau }\left( \boldsymbol{\theta }_{0}\right) \boldsymbol{P}_{\tau }\left( \boldsymbol{\theta }_{0}\right) \\
\boldsymbol{P}_{\tau }\left( \boldsymbol{\theta }_{0}\right) & = & \boldsymbol{J}_{\tau }\left( \boldsymbol{\theta }_{0}\right) ^{-1}-\boldsymbol{Q}_{\tau }\left( \boldsymbol{\theta }_{0}\right) \boldsymbol{M}_{\tau }^{T}\left( \boldsymbol{\theta }_{0}\right) \boldsymbol{J}_{\tau }\left( \boldsymbol{\theta }_{0}\right) ^{-1} \\
\boldsymbol{Q}_{\tau }\left( \boldsymbol{\theta }_{0}\right) & = & \boldsymbol{J}_{\tau }\left( \boldsymbol{\theta }_{0}\right) ^{-1}{}_{\tau }\left( \boldsymbol{\theta }_{0}\right) \left[ \boldsymbol{M}_{\tau }\left(
\boldsymbol{\theta }_{0}\right) ^{T}\boldsymbol{J}_{\tau }\left( \boldsymbol{\theta }_{0}\right) ^{-1}\boldsymbol{M}_{\tau }\left( \boldsymbol{\theta }_{0}\right) \right] ^{-1}.
\end{eqnarray*}
\end{enumerate}

Therefore, following similar steps as in Theorem (\ref{thm:score}), it can be established that the score vector at the RMDPDE satisfies:
\begin{equation*}
\sqrt{n}\mathbf{U}_{\tau }\left( \widetilde{\alpha }_{\tau },\widetilde{\beta }_{\tau }\right) \overset{\mathcal{L}}{\underset{n\rightarrow \infty }{\longrightarrow }}\mathcal{N}\left( \boldsymbol{0},\boldsymbol{K}_{\tau
}\left( \boldsymbol{\theta }_{0}\right) \right).
\end{equation*}

\begin{definition}
	The Rao-type test statistic based on the MDPDE with tuning parameter $\tau$ for testing the composite null hypothesis (\ref{W1.5_2}) is given by
\begin{equation}
	R_{n}\left( \widetilde{\boldsymbol{\theta }}_{\tau }\right) = n\mathbf{U}_{\tau }\left( \widetilde{\boldsymbol{\theta }}_{\tau }\right) ^{T} \boldsymbol{Q}_{\tau }\left( \widetilde{\boldsymbol{\theta }}_{\tau }\right) \left[ \boldsymbol{Q}_{\tau }\left( \widetilde{\boldsymbol{\theta }}_{\tau }\right) ^{T}\boldsymbol{K}_{\tau }\left( \widetilde{\boldsymbol{\theta }}_{\tau }\right) ^{-1}\boldsymbol{Q}_{\tau }\left( \widetilde{\boldsymbol{\theta }}_{\tau }\right) \right] ^{-1}\boldsymbol{Q}_{\tau }\left( \widetilde{\boldsymbol{\theta }}_{\tau }\right) ^{T}\mathbf{U}_{\tau }\left( \widetilde{\boldsymbol{\theta }}_{\tau }\right)
\end{equation}
where $\mathbf{U}_{\tau }\left( \alpha ,\beta \right) =\left( U_{\tau }\left( \alpha \right) ,U_{\tau }\left( \beta \right) \right) ^{T}$ 
and 
$\boldsymbol{K}_{\tau }(\alpha ,\beta )$ is given in (\ref{W1.1421}).
\end{definition}
(see \cite{baghmapa22} for more details).

Consequently, the Rao-type test statistic based on the MDPDE with tuning paramter $\tau$ follows asymptotically a chi-square distribution with $r$ degrees of freedom.
\begin{equation*}
R_{n}\left( \widetilde{\boldsymbol{\theta }}_{\tau }\right) 
\overset{\mathcal{L}}{\underset{n\rightarrow \infty }{\longrightarrow }}\chi _{r}^{2}
\end{equation*}
and we must reject the null hypothesis (\ref{W1.5_2}) if
\begin{equation*}
R_{n}\left( \widetilde{\boldsymbol{\theta }}_{\tau }\right) >\chi _{r,\alpha }^{2}.
\end{equation*}

\section{Simulation Study}

To evaluate the performance of this new families of test statistics, we conduct an extensive simulation study.
We have consider a log-logistic distribution with true scale and shape parameters $\alpha =1$ and $\beta =5.$ 
To evaluate the performance of the test statistics based on the MDPDE, we consider a simple null hypothesis on the scale parameter $\alpha$, and assume the shape $\beta =5$ is known. 

The simulation is conducted as follows: first, we generate a sample of size $n$ of a log-logistic with parameters $(\alpha, \beta) = (1, 5).$ The sample size ranges from $n = 20, 30, ..., 100.$ 
For each generated sample, we obtain the corresponding MDPDEs $\widehat{\alpha}_{\tau}$ of $\alpha$ for different values of the tuning parameter $\tau$ ranging from $\tau=0, 0.1, ..., 1.$ Based on the MDPDEs, we compute the Wald-type and Rao-type test statistics $W_{n,\tau}(\alpha ),$  given by
$$ W_{n,\tau }(\hat{\alpha}_\tau)= n (\hat{\alpha}_\tau -\alpha_0)^2 \frac{ J^{11}_{\tau}(\alpha_0, 5)^2}{ K^{11}_{\tau}(\alpha_0,5)},$$ 
and
$$ R_{n,\tau }(\hat{\alpha}_\tau)= K^{11}_{\tau}(\alpha_0, 5)^{-1}\left\{\sum_{i=1}^{n} \frac{5(-\alpha_0^5+X_i^5)5^\tau \alpha_0^{5\tau}X_i^{4\tau}}{\alpha_0(\alpha_0^5+X_i^5)^{2\tau+1}} + \frac{\tau}{\tau+1}\frac{\beta^5}{\alpha_0^6} B\left(\frac{6\tau + 5}{5}, \frac{4\tau+5}{5}\right) \right\},$$ 
where $J^{11}_{\tau}(\alpha, \beta)$ and $K^{11}_{\tau}(\alpha, \beta)$ are defined in (\ref{W1.1421}).


To assess the empirical significance level, we consider the following null hypothesis
\begin{equation}\label{simplenocontcase1}
H_0: \alpha =\alpha_0 :=1, H_1: \alpha \ne 1,
\end{equation}

Moreover, to evaluate the performance of the test statistic under contamination, we introduce an $\varepsilon\%$ of outlying observations on the sample, with the contamination proportion ranging from $\varepsilon = 0\%,...,20\%.$ Outlying observations are generated from a log-logistic distribution with true scale parameter $\widetilde{\alpha} = 3, 6$ and the population shape $\beta = 5.$

At a significance level of $\alpha = 5\%$, the  test  based on the Wald-type statistic with tuning parameter $\tau$ rejects the null hypothesis if $W_{n,\tau }(\widehat{\alpha}_\tau )> \chi^2_{1;0.05},$ and similarly, the test based on the Rao-type statistic with tuning parameter $\tau$ rejects the null hypothesis if $R_{n,\tau }(\widehat{\alpha}_\tau )> \chi^2_{1;0.05},$
We repeat the simulation for $R=10.000$ times and we compute the empirical level as the averaged number of rejections.


%
%
%


Figures \ref{fig:sample_size} and \ref{fig:sample_size_rao} show the empirical level of the Wald-type and Rao-type test statistics for a grid of $\tau$ values in the absence of contamination (left) and under a $15\%$ of outlying data for increasing sample size. In the absence of contamination, all test statistics exhibit similar performance and converge to the nominal significance level as the sample size increases. However, for small sample sizes, the superiority of the MLE becomes more evident. In contrast, in the presence of outlying observations, the MLE is significantly affected, with its deterioration becoming more pronounced as the sample size increases. Remarkably, the empirical level of the classical Wald test is always over 0.4 (rejecting the null hypothesis) and reaches high values around 0.9.
 On the other hand, the robust MDPDEs with moderate and high values of the tuning parameter, above $\tau=0.3,$ keep competitive with low empirical level under contamination, thus demonstrating the advantage of the robust test in contaminated scenarios.  

\begin{figure}[htb]
	\begin{subfigure}[c]{0.53\textwidth}
		\includegraphics[scale=0.28]{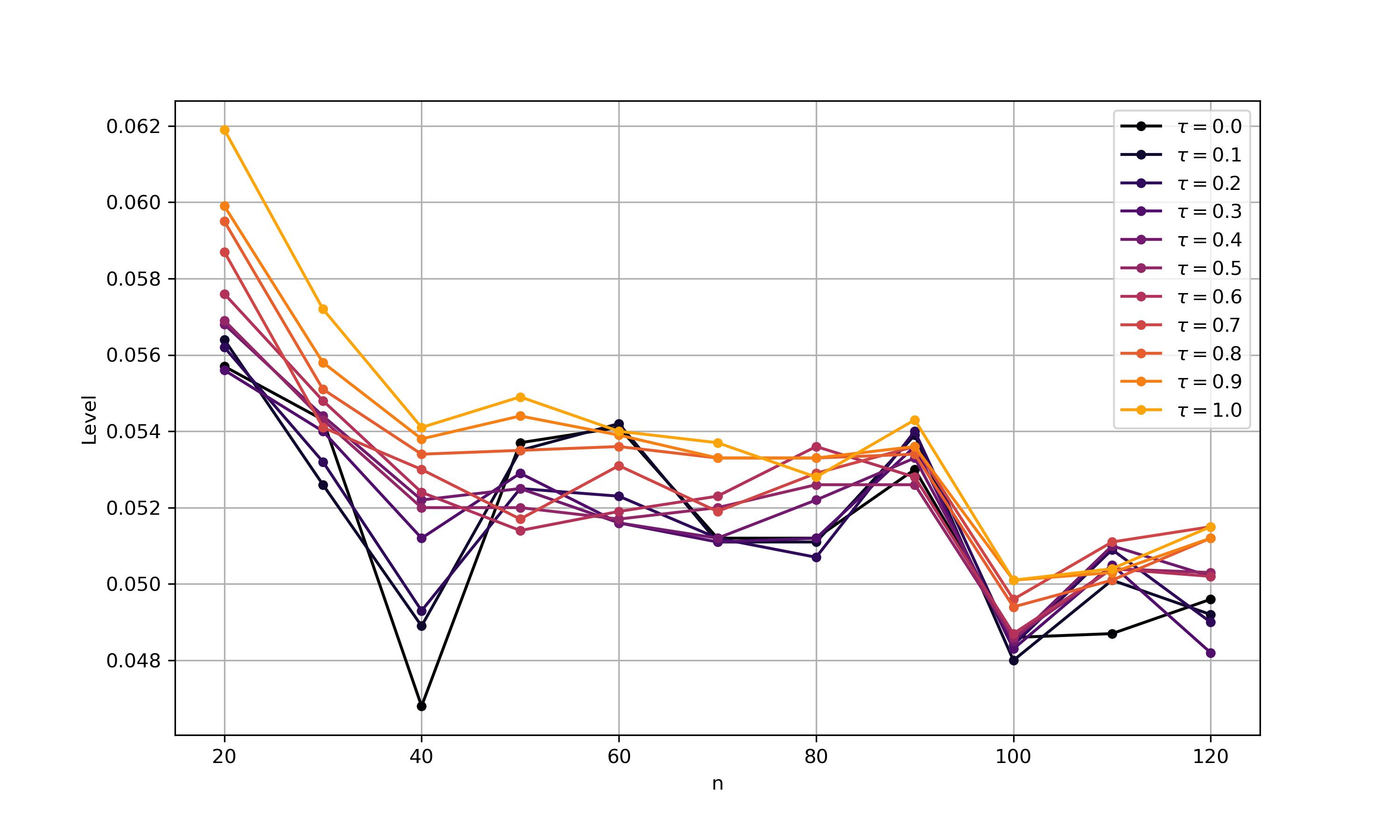}
		\subcaption{No contamination}
	\end{subfigure}
	\begin{subfigure}[c]{0.48\textwidth}
		\includegraphics[scale=0.28]{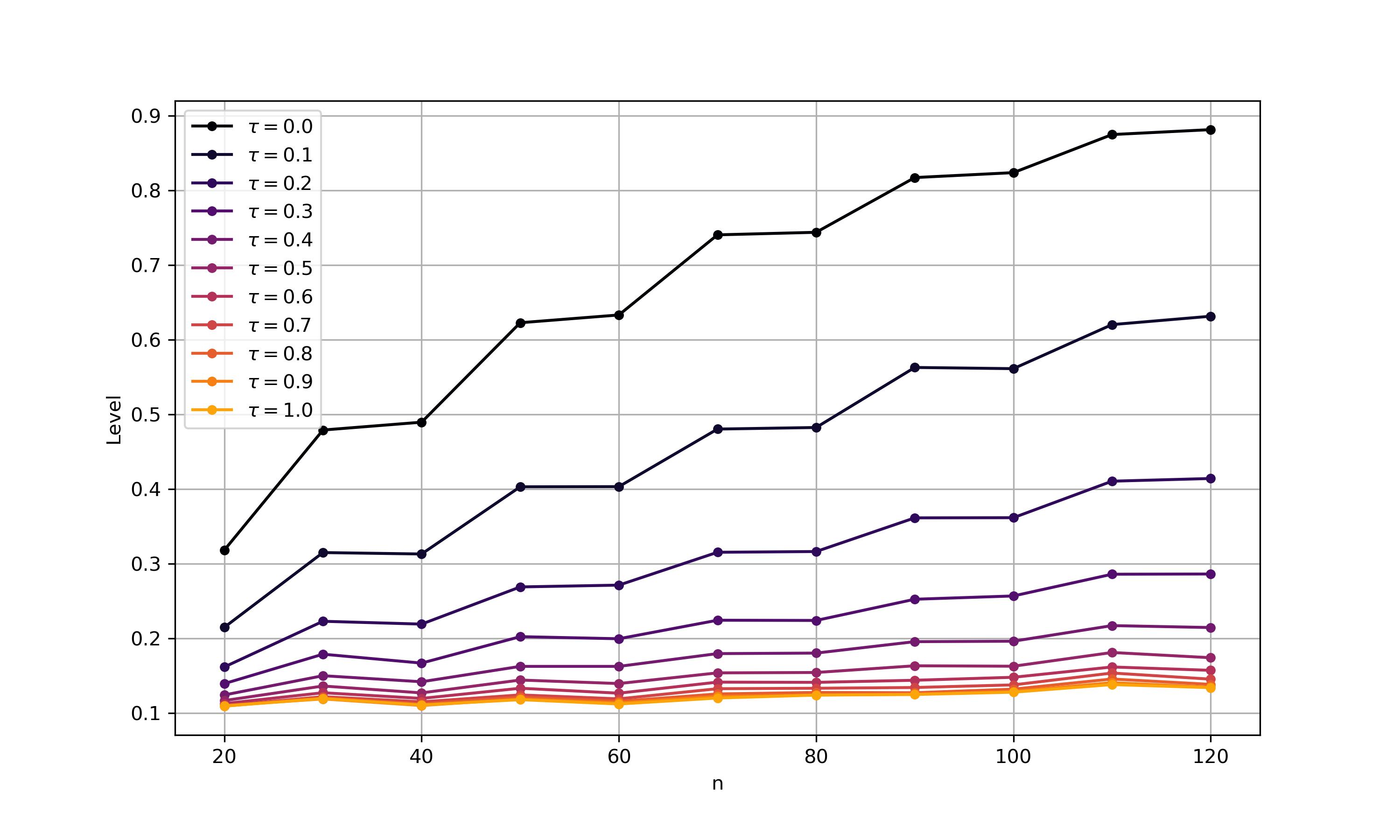}
		\subcaption{$15\%$ contamination}
	\end{subfigure}
	\caption{Empirical level of the Wald-type tests under increasing sample size for different values of $\tau$ in the absence of contamination (left) and with a $15\%$ contamination proportion (right) and $\widetilde{\alpha} = 3$.}
	\label{fig:sample_size}
\end{figure}

\begin{figure}[htb]
	\begin{subfigure}[c]{0.53\textwidth}
		\includegraphics[scale=0.28]{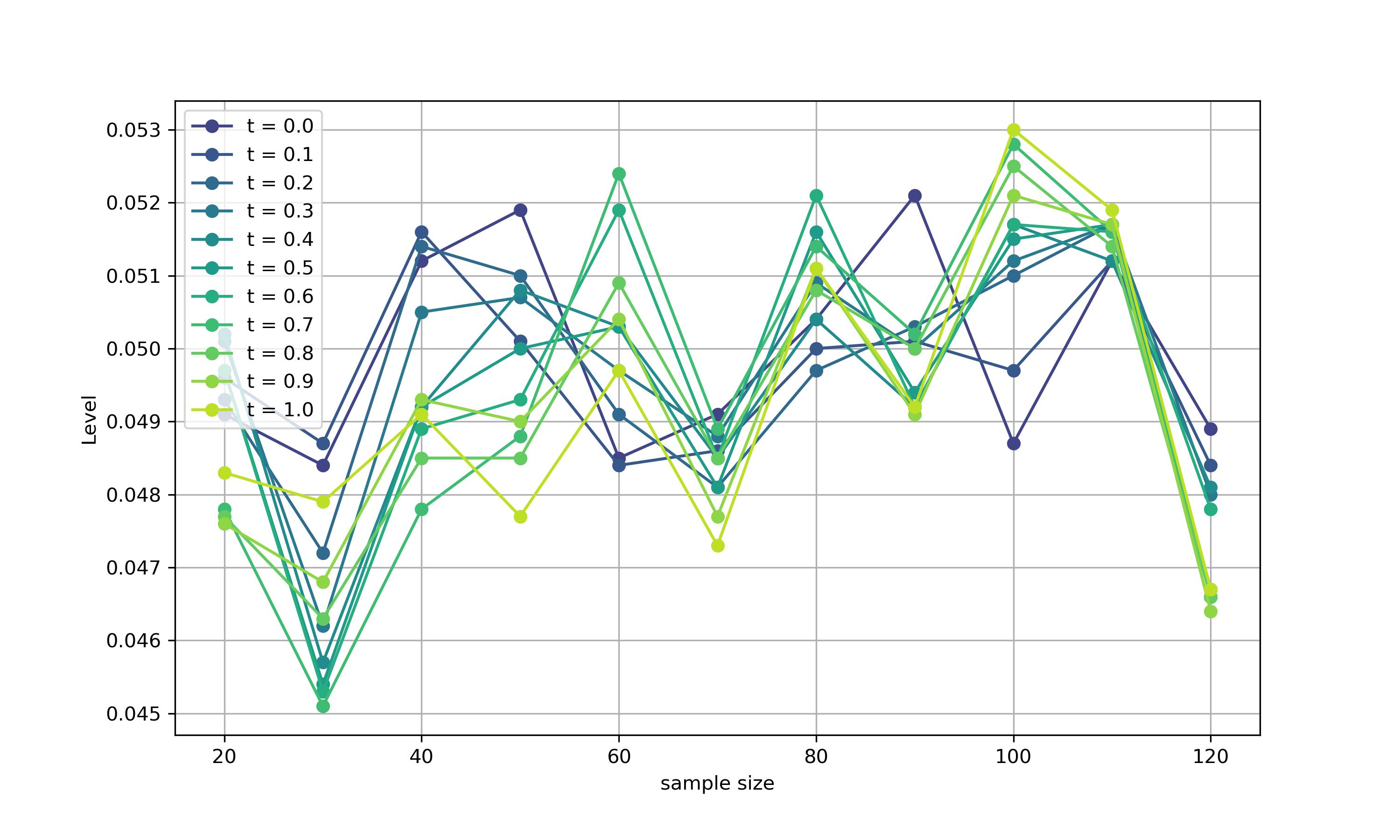}
		\subcaption{No contamination}
	\end{subfigure}
	\begin{subfigure}[c]{0.48\textwidth}
		\includegraphics[scale=0.28]{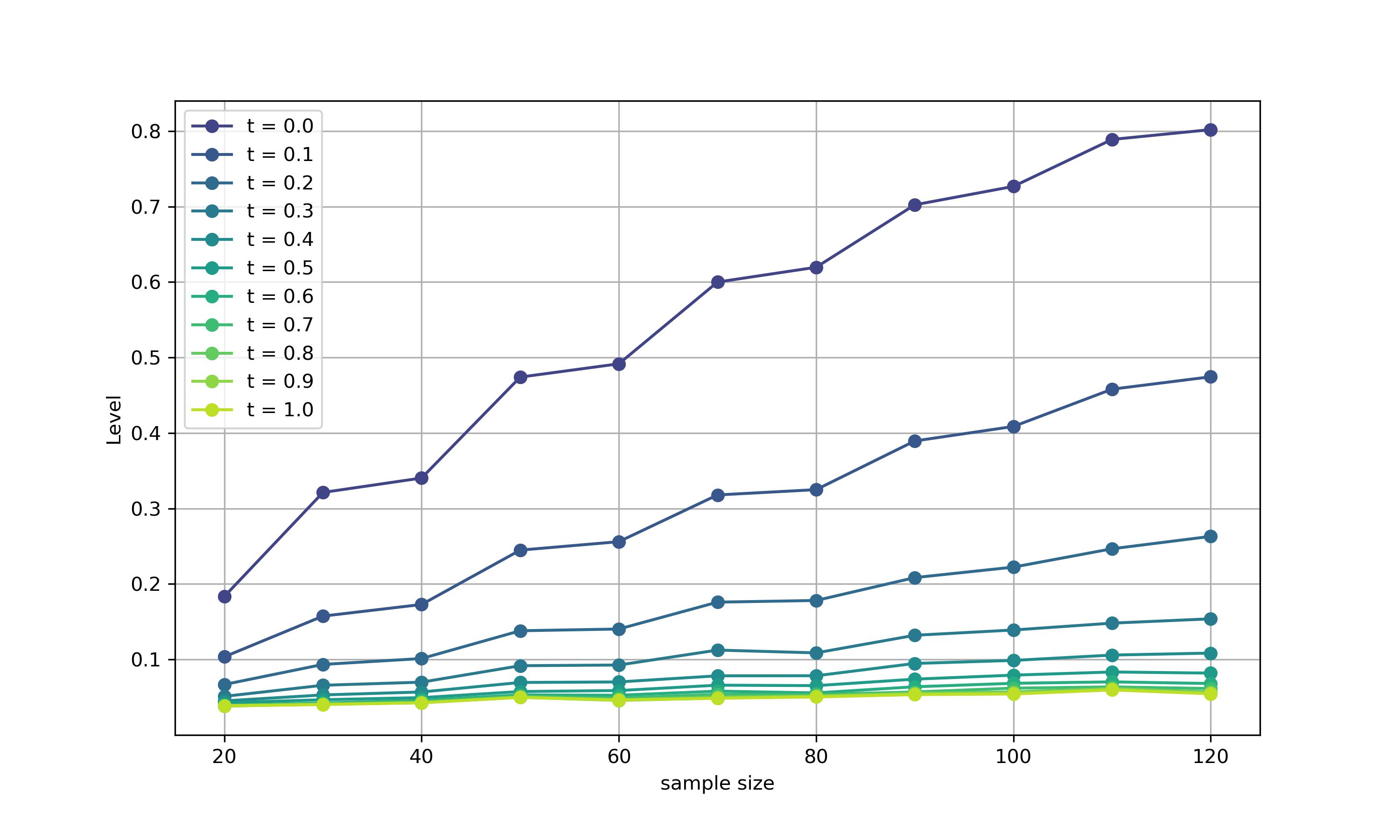}
		\subcaption{$15\%$ contamination}
	\end{subfigure}
	\caption{Empirical level of the Rao-type tests under increasing sample size for different values of $\tau$ in the absence of contamination (left) and with a $15\%$ contamination proportion (right) and $\widetilde{\alpha} = 3$.}
	\label{fig:sample_size_rao}
\end{figure}

Figures \ref{fig:levelcont3} and \ref{fig:levelcont3_rao} plots the empirical level of the test against sample contamination for the same grid of $\tau$ values and sample size of $n=50$ (left) and $n=100$ (right), where the outlying observations where generated from a log-logistic distribution with scale $\widehat{\alpha} = 3$ (moderate outliers). The same plots but setting $\widehat{\alpha} = 6$ (strong outliers) are displayed in figures \ref{fig:levelcont6} and \ref{fig:levelcont6_cont}.
 The classical test based on the MLE (corresponding to $\tau=0$) gets rapidly affected by sample contamination, and even a small amount of contamination can alter the test decision significantly. In particular, above $5\%$ the empirical level exceeds  0.15. Moreover, as the proportion of outliers increases, the empirical level continues to rise, leading to a higher likelihood of incorrectly rejecting the null hypothesis. Conversely, the Wald-type and Rao-type test statistics based on MDPDEs with moderate and high values of the tuning parameter remain slightly affected by this outlying observations and result in much smaller empirical levels.

\begin{figure}[htb]
	\begin{subfigure}[c]{0.53\textwidth}
		\includegraphics[scale=0.28]{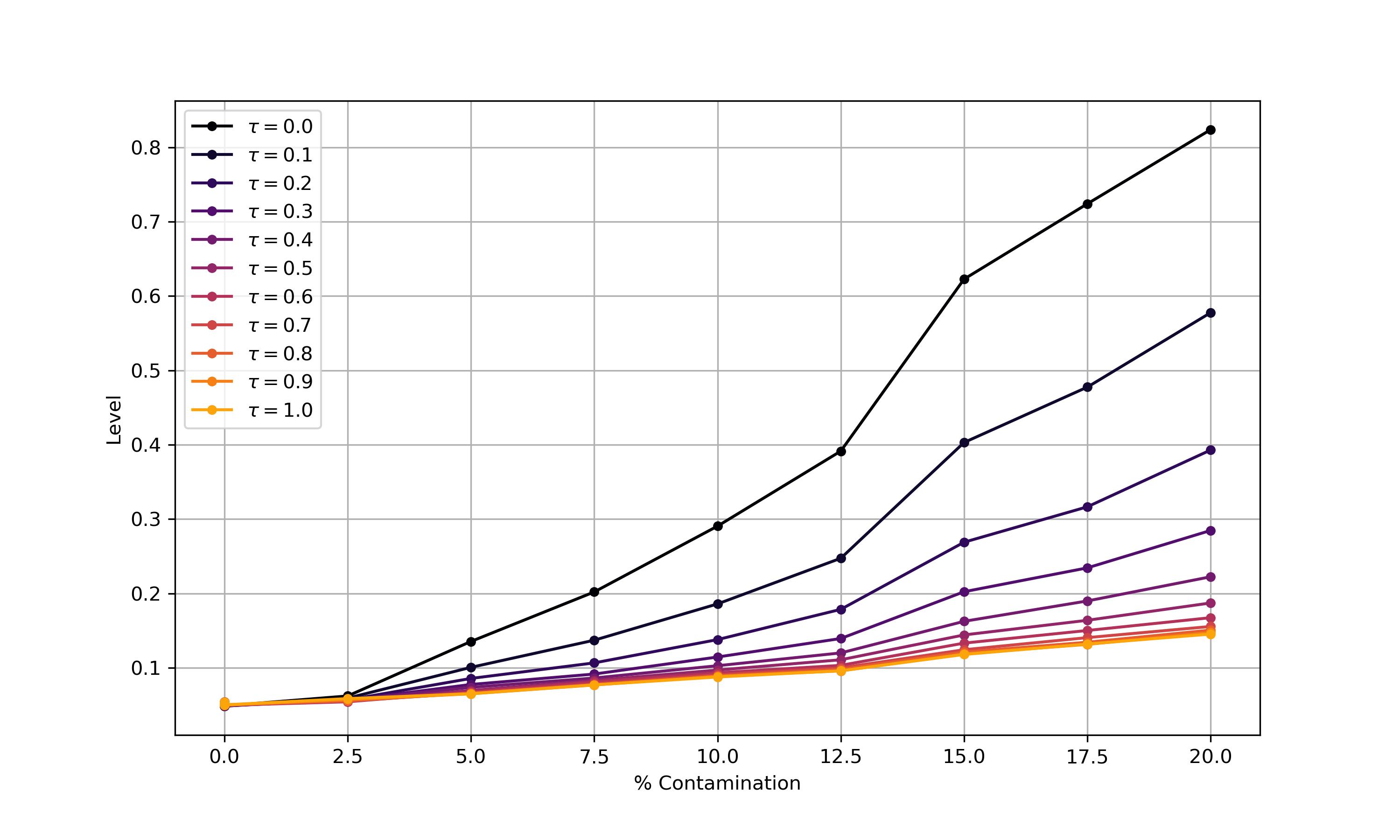}
		\subcaption{$n=50$}
	\end{subfigure}
	\begin{subfigure}[c]{0.48\textwidth}
		\includegraphics[scale=0.28]{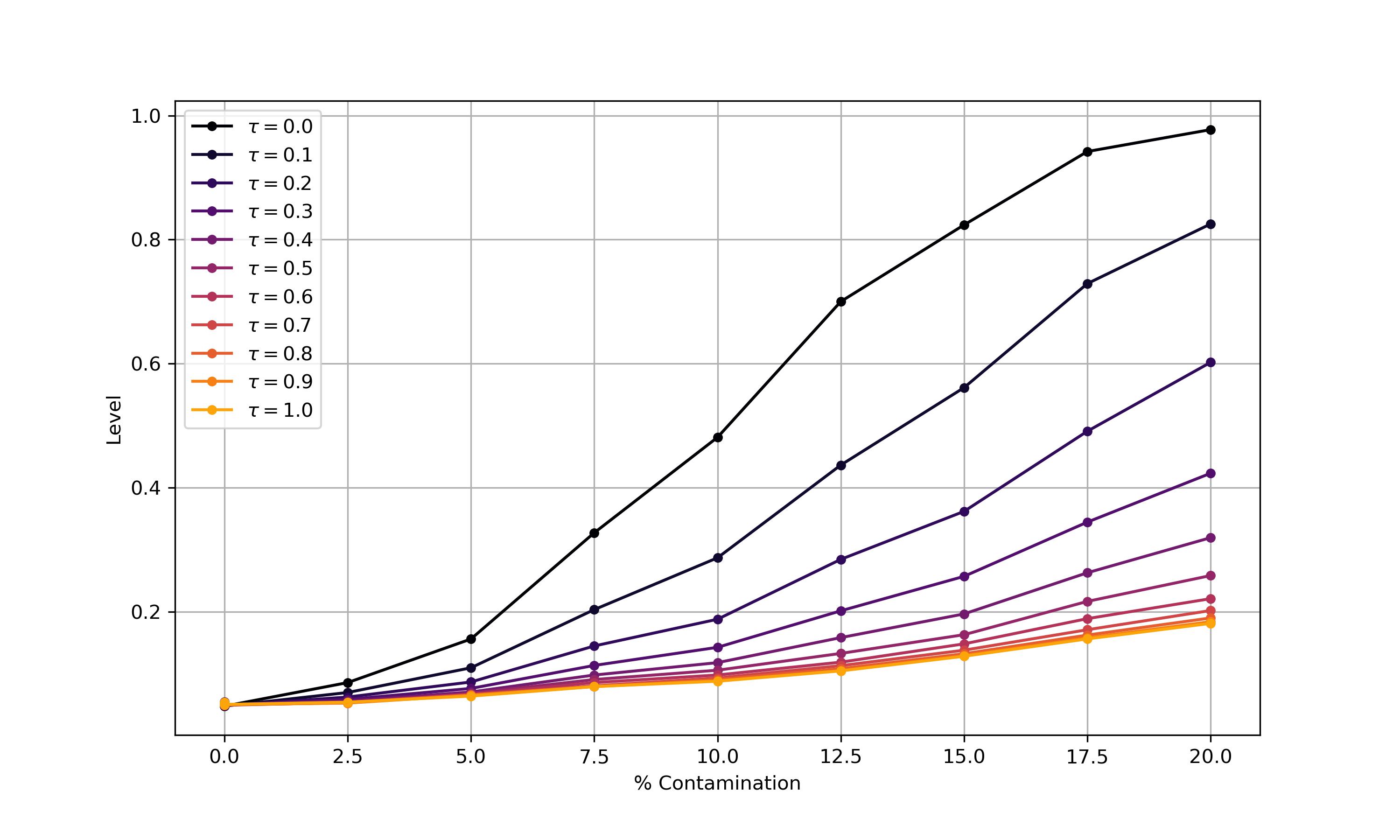}
		\subcaption{$n=100$}
	\end{subfigure}
	\caption{Empirical level of the Wald-type test statistics with $n=50$ (left) and $n=100$ (rigth) under increasing contamination proportion with $\widetilde{\alpha } = 3$ under different values of $\tau$.}
	\label{fig:levelcont3}
\end{figure}

\begin{figure}[htb]
	\begin{subfigure}[c]{0.53\textwidth}
		\includegraphics[scale=0.28]{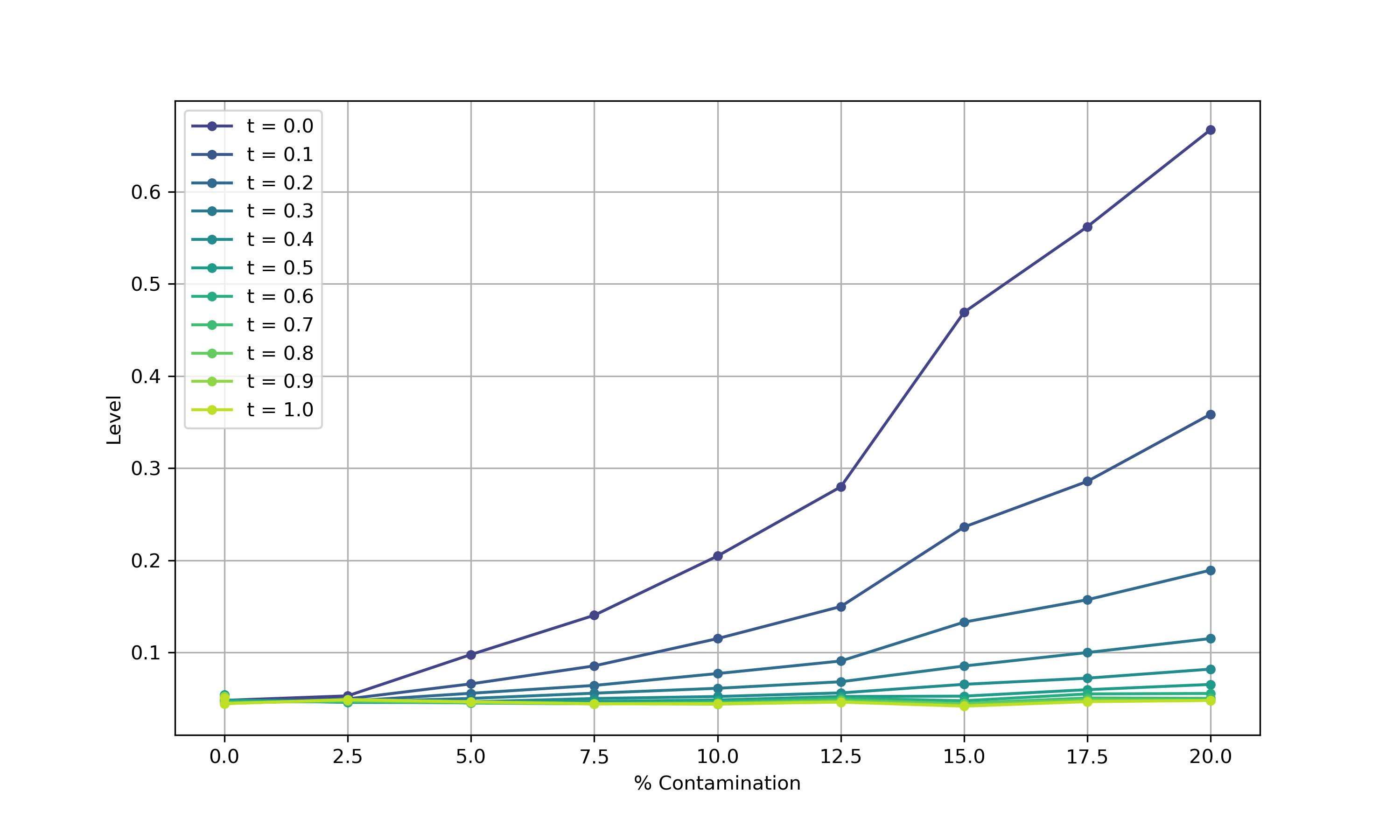}
		\subcaption{$n=50$}
	\end{subfigure}
	\begin{subfigure}[c]{0.48\textwidth}
		\includegraphics[scale=0.28]{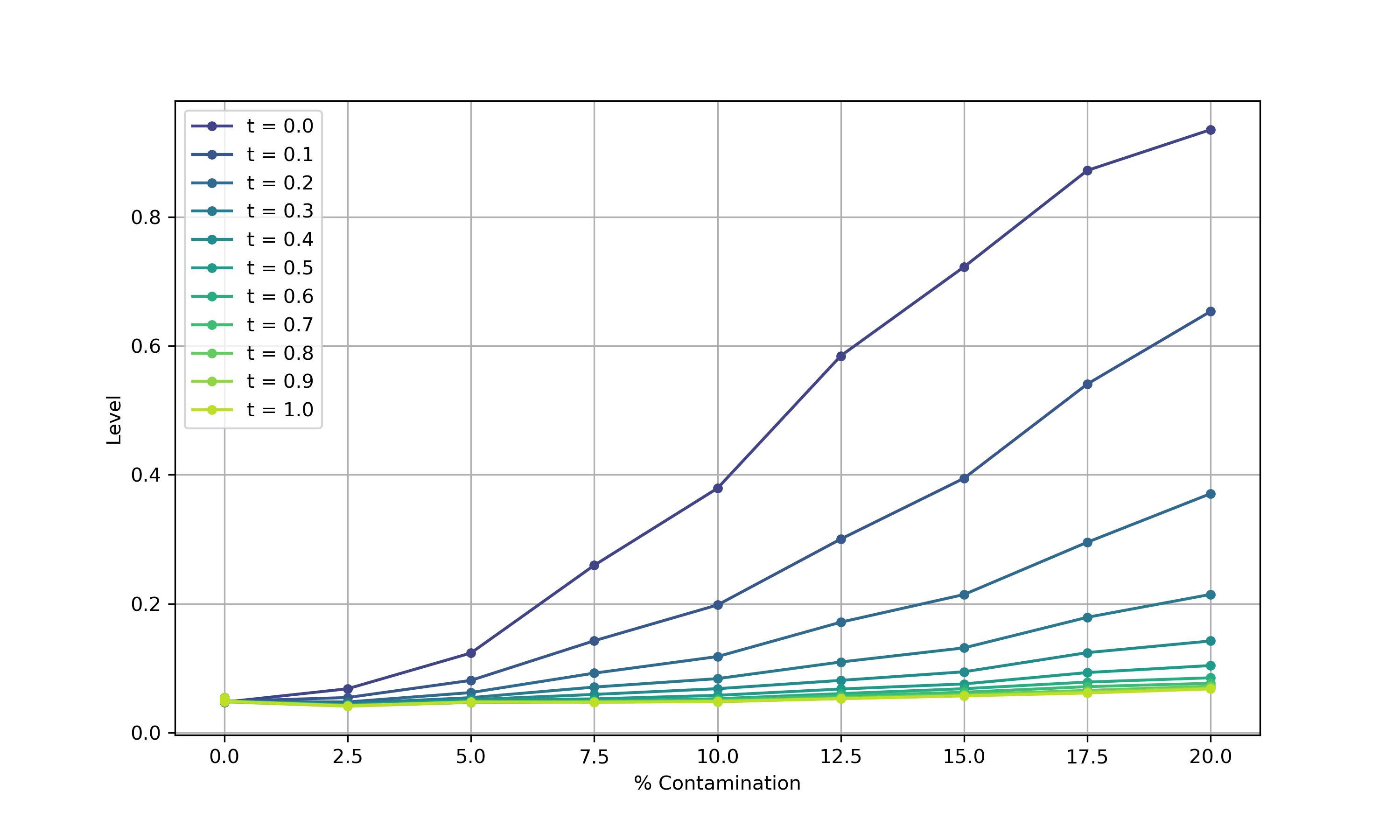}
		\subcaption{$n=100$}
	\end{subfigure}
	\caption{Empirical level of the Rao-type test statistics with $n=50$ (left) and $n=100$ (rigth) under increasing contamination proportion with $\widetilde{\alpha } = 3$ under different values of $\tau$.}
	\label{fig:levelcont3_rao}
\end{figure}

\begin{figure}[htb]
	\begin{subfigure}[c]{0.53\textwidth}
		\includegraphics[scale=0.28]{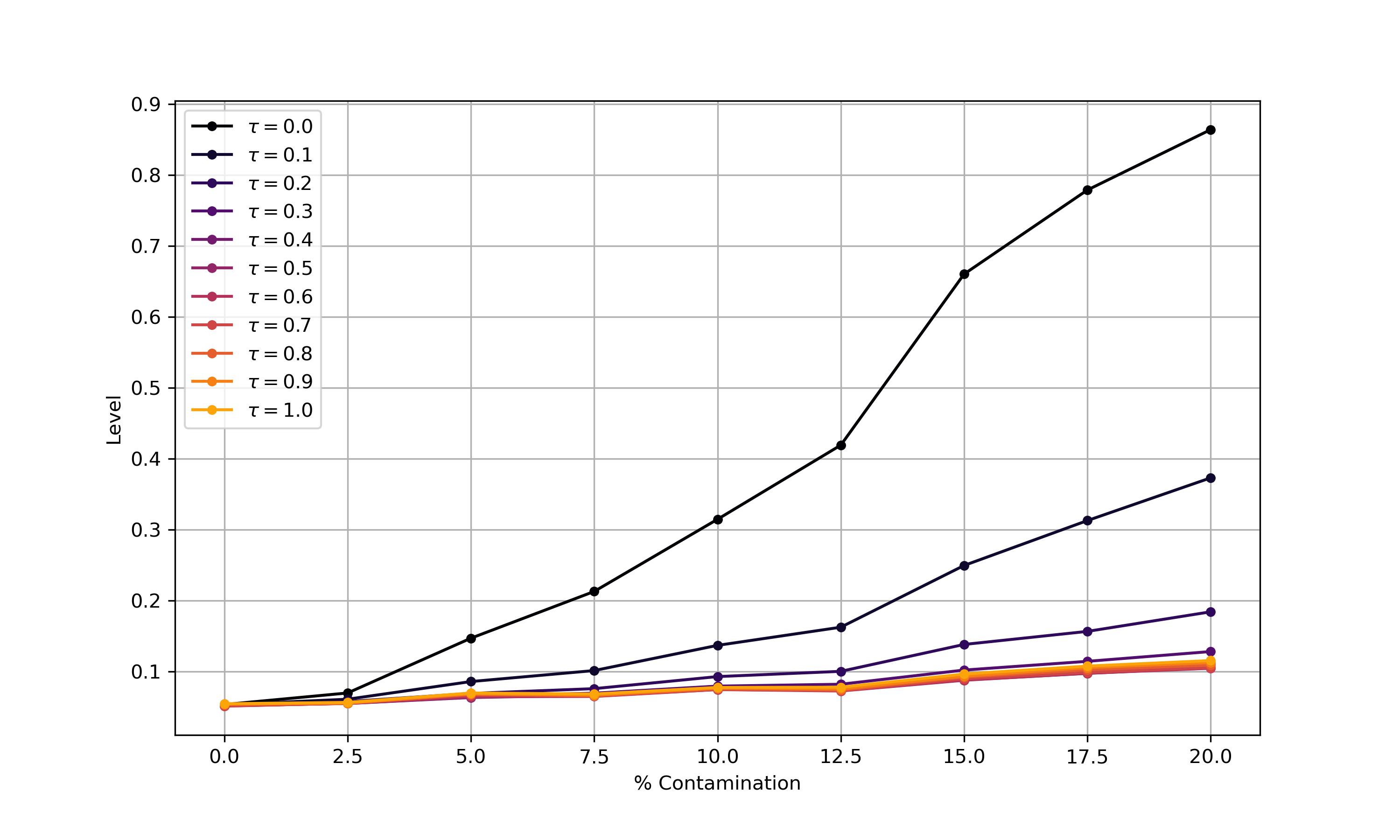}
		\subcaption{$n=50$}
	\end{subfigure}
	\begin{subfigure}[c]{0.48\textwidth}
		\includegraphics[scale=0.28]{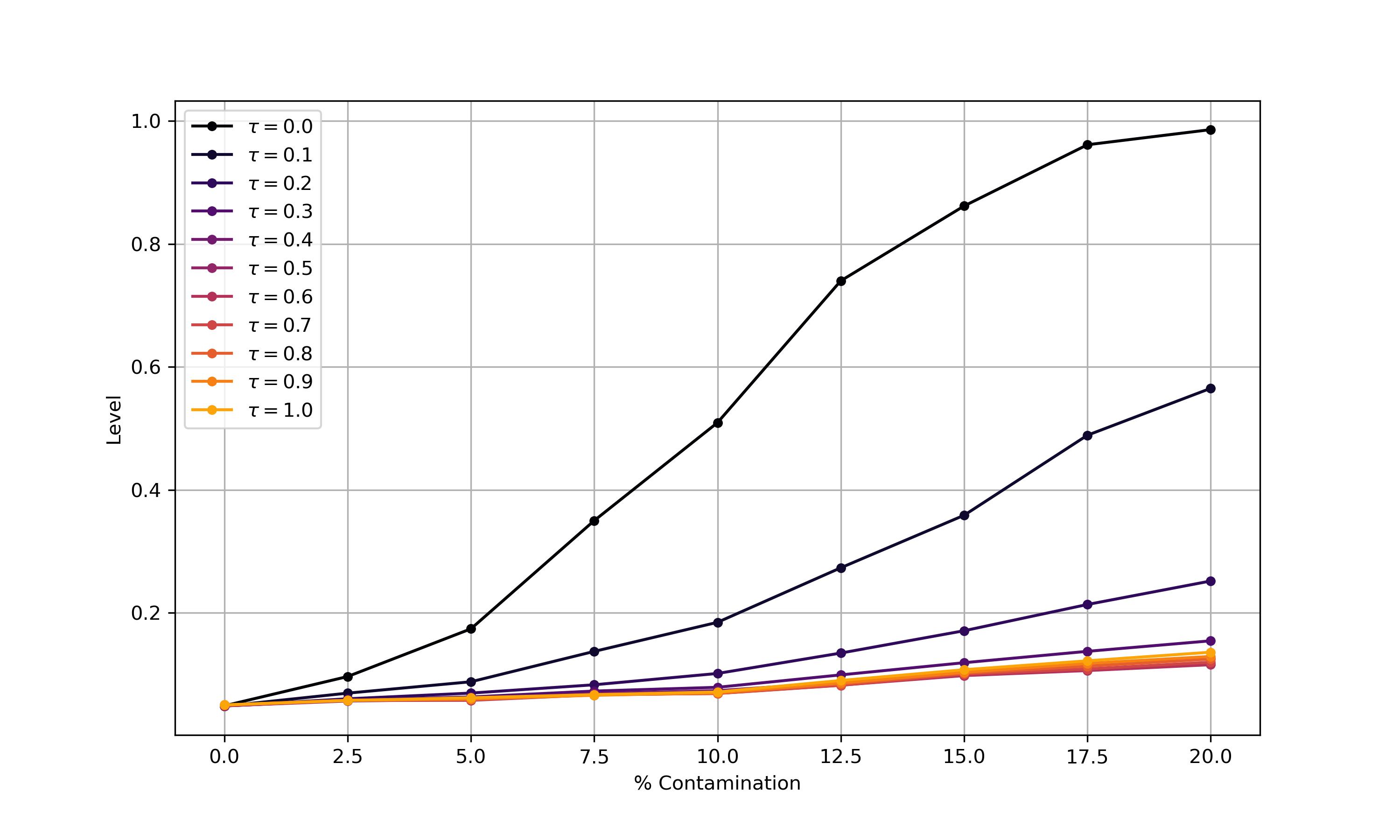}
		\subcaption{$n=100$}
	\end{subfigure}
	\caption{Empirical level of the Wald-type test statistics with $n=50$ (left) and $n=100$ (rigth) under increasing contamination proportion with $\widetilde{\alpha } = 6$ under different values of $\tau$.}
	\label{fig:levelcont6}
\end{figure}

\begin{figure}[htb]
	\begin{subfigure}[c]{0.53\textwidth}
		\includegraphics[scale=0.28]{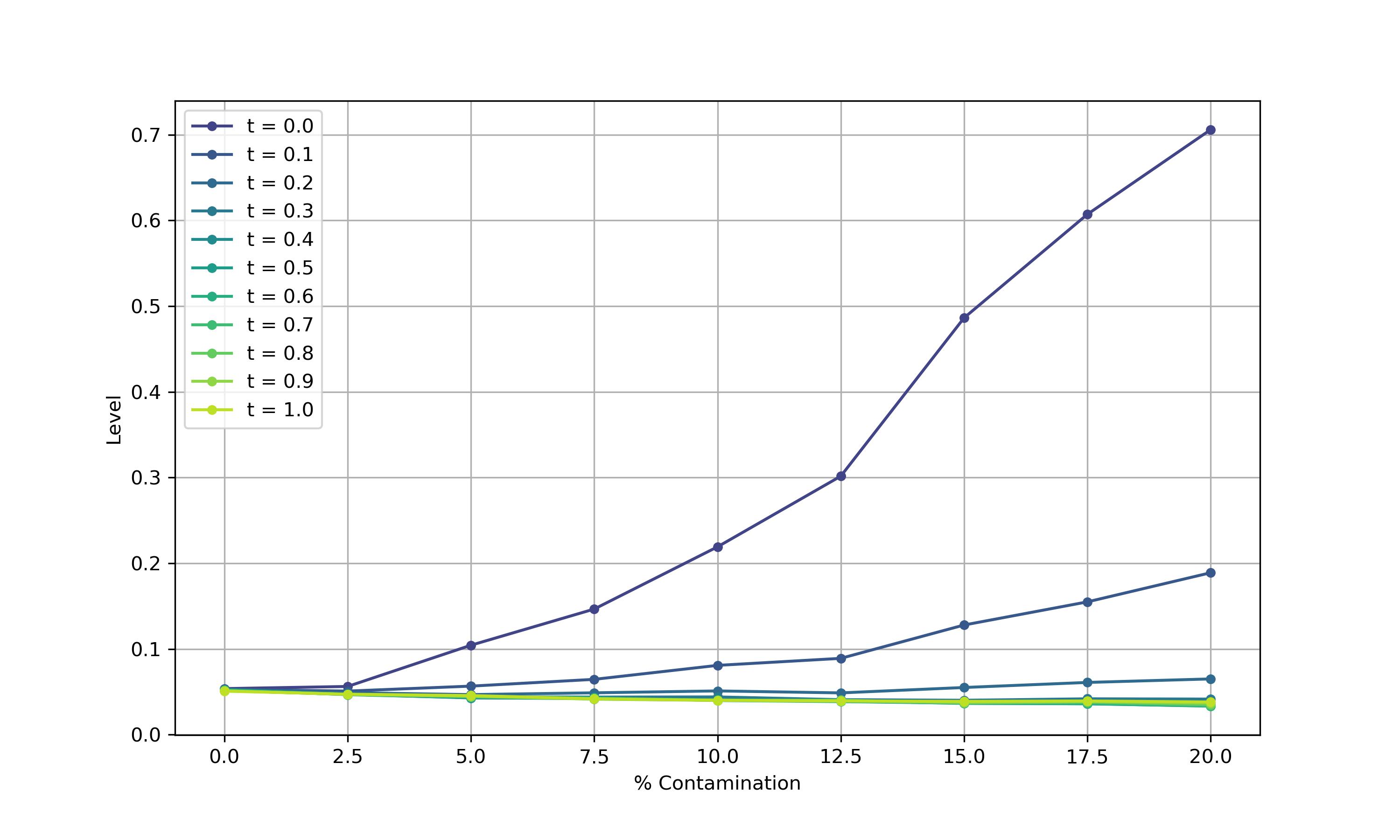}
		\subcaption{$n=50$}
	\end{subfigure}
	\begin{subfigure}[c]{0.48\textwidth}
		\includegraphics[scale=0.28]{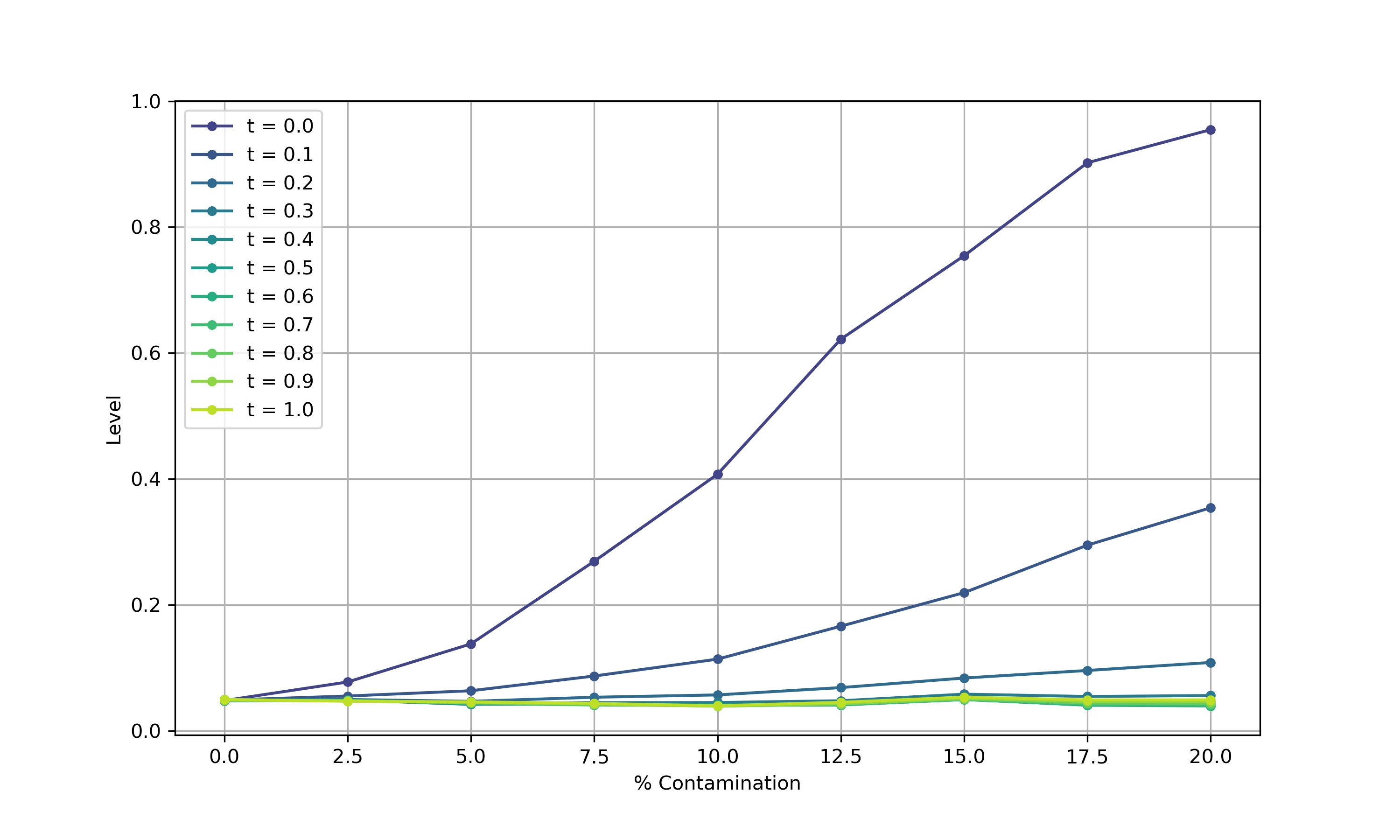}
		\subcaption{$n=100$}
	\end{subfigure}
	\caption{Empirical level of the Rao-type test statistics with $n=50$ (left) and $n=100$ (rigth) under increasing contamination proportion with $\widetilde{\alpha } = 6$ under different values of $\tau$.}
	\label{fig:levelcont6_rao}
\end{figure}

A natural question that may arise is the election of the test statistics family in practice. As previously discussed, there is no universally superior choice for general models, as each family presents its own advantages and disadvantages. Our results indicate that both families exhibit similar behavior in terms in robustness. 

Figure \ref{fig:levelcont_both} plots a comparative performance of both families under increasing contamination with moderate outliers ($\widetilde{\alpha} = 3$, left) and strong outliers ($\widetilde{\alpha} = 6$, right). The sample size is set to $n=100$ in both graphics. The classical Wald and Rao tests ($\tau=0$) perform similarly and get rapidly affected by contamination. Moreover, the Rao family tends to be more conservative, exhibiting lower empirical levels compared to the Wald family. Nonetheless, the robustness of both families is clearly demonstrated.

\begin{figure}[htb]
	\begin{subfigure}[c]{0.53\textwidth}
		\includegraphics[scale=0.28]{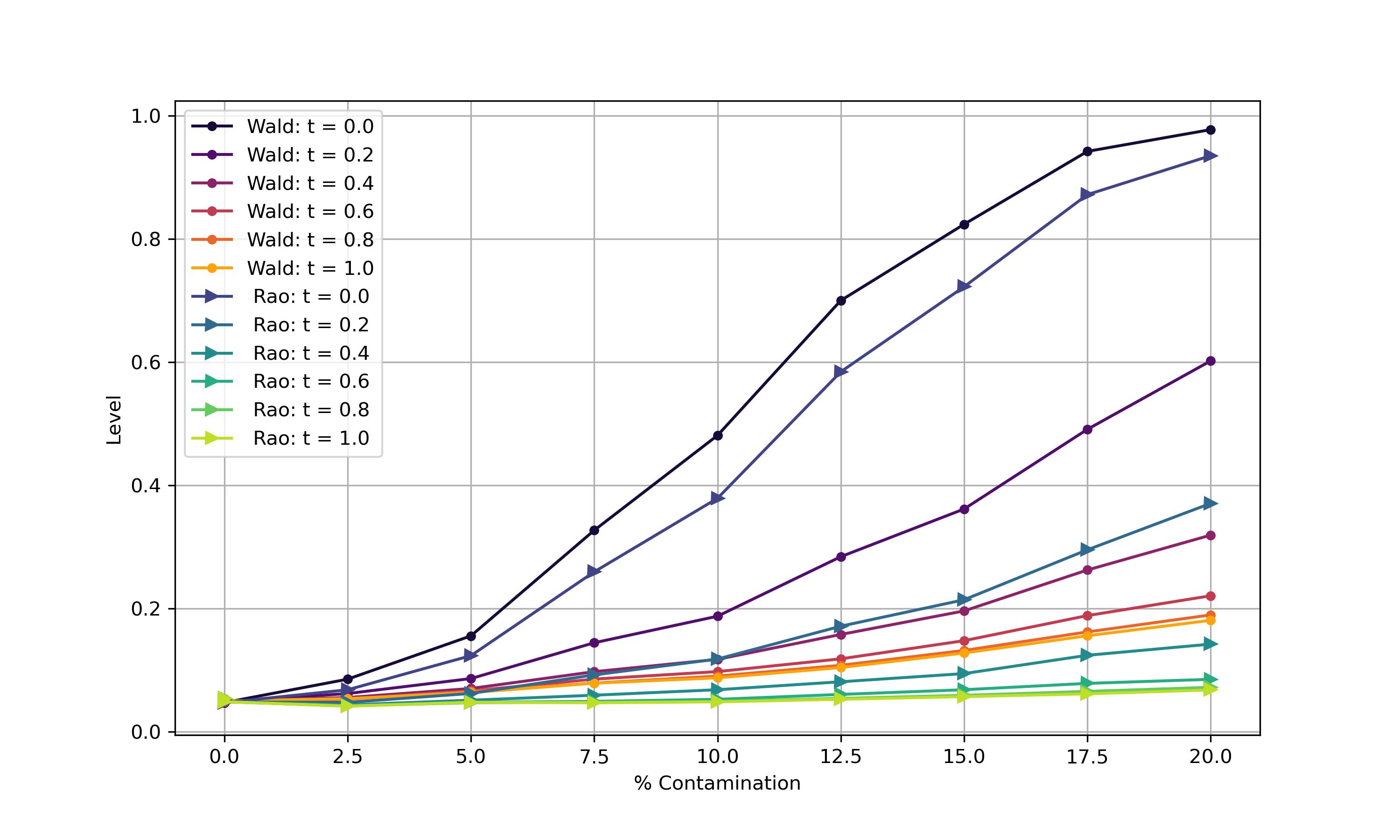}
		\subcaption{$\widetilde{\alpha} = 3$}
	\end{subfigure}
	\begin{subfigure}[c]{0.48\textwidth}
		\includegraphics[scale=0.28]{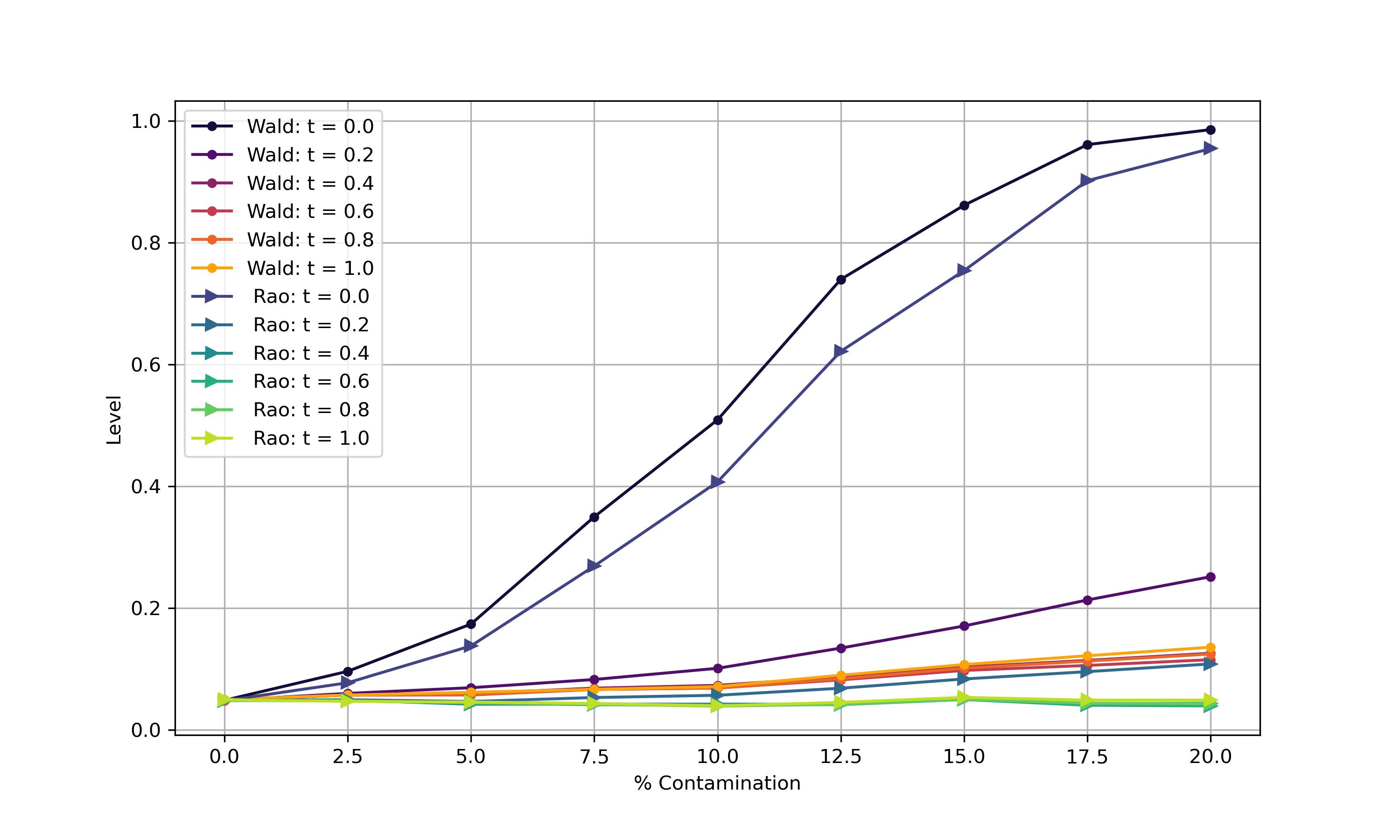}
		\subcaption{$\widetilde{\alpha} = 6$}
	\end{subfigure}
	\caption{Empirical level of both families with $n=50$ (left) and $n=100$ (rigth) under increasing contamination proportion with $\widetilde{\alpha } = 6$ under different values of $\tau$.}
	\label{fig:levelcont_both}
\end{figure}


Next, we study the power of the test we consider again the tests given in (\ref{simplenocontcase1})
Now the underlying distribution is a log-logistic distribution with true parameters $(\alpha, \beta) = (1.15, 5)$ and contaminated scale parameter $\widetilde{\alpha} = 0.5$

Figures \ref{fig:powern} and \ref{fig:powern_rao} present the empirical power of the tests in the absence of contamination (left) and under a $15\%$ of contamination (right) against sample size for a grid of values of the tuning parameter $\tau.$ The results shows the stated consistency of the test in the absence of contamination, as all empirical powers tends to 1 when increasing the sample size. In the right panel, the power of the tests is impacted by the presence of outlying observations, with the effect becoming more pronounced as the value of $\tau$ increases. The role of $\tau$ controlling the compromise between efficiency and robustness is evidenced; the classical Wald test has the best performance in the absence of contamination, but gets heavily affected by data contamination. On the other hand, moderate values of the tuning parameter $\tau$ provide robust test statistics with a small loss in efficiency.

\begin{figure}[htb]
	\begin{subfigure}[c]{0.53\textwidth}
		\includegraphics[scale=0.28]{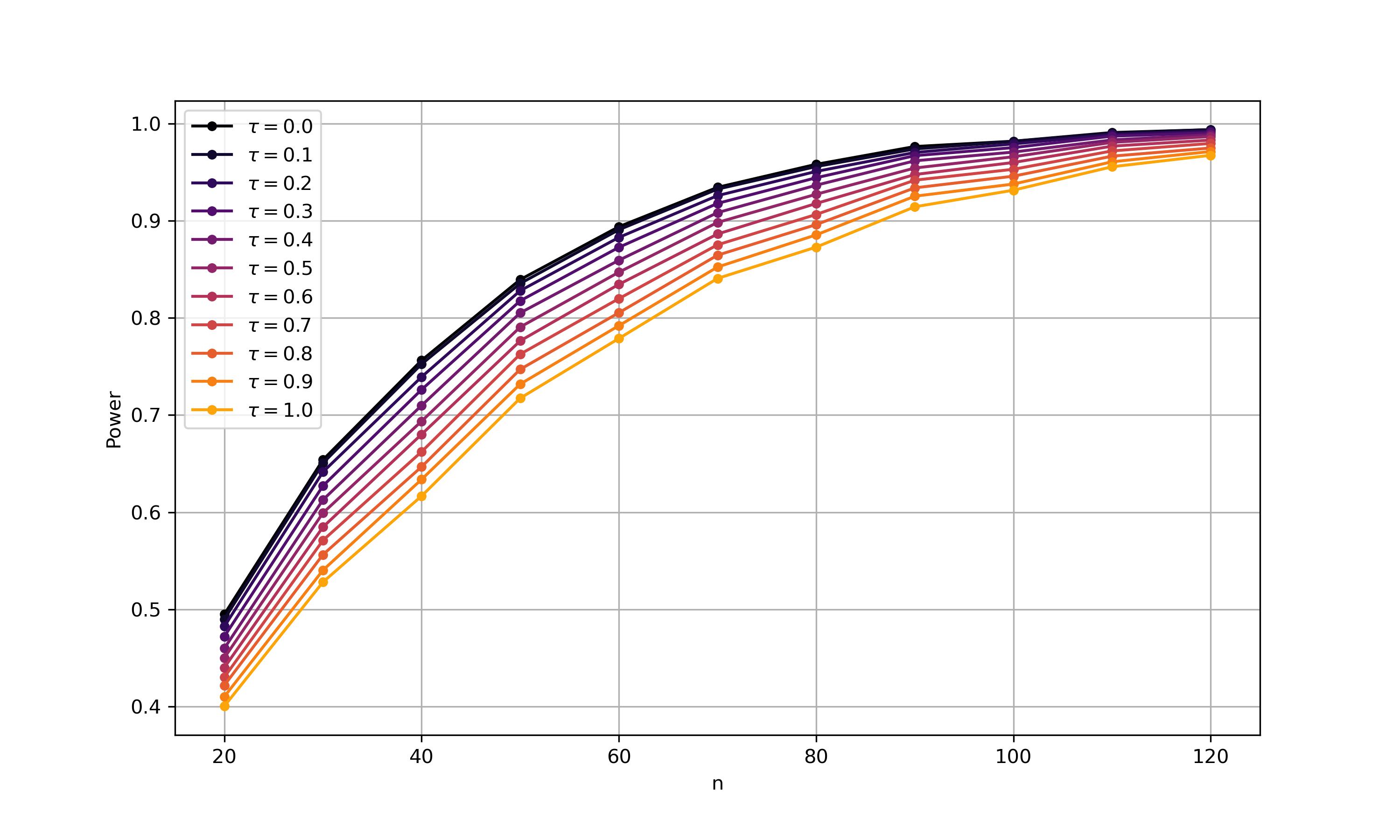}
		\subcaption{No contamination}
	\end{subfigure}
	\begin{subfigure}[c]{0.48\textwidth}
		\includegraphics[scale=0.28]{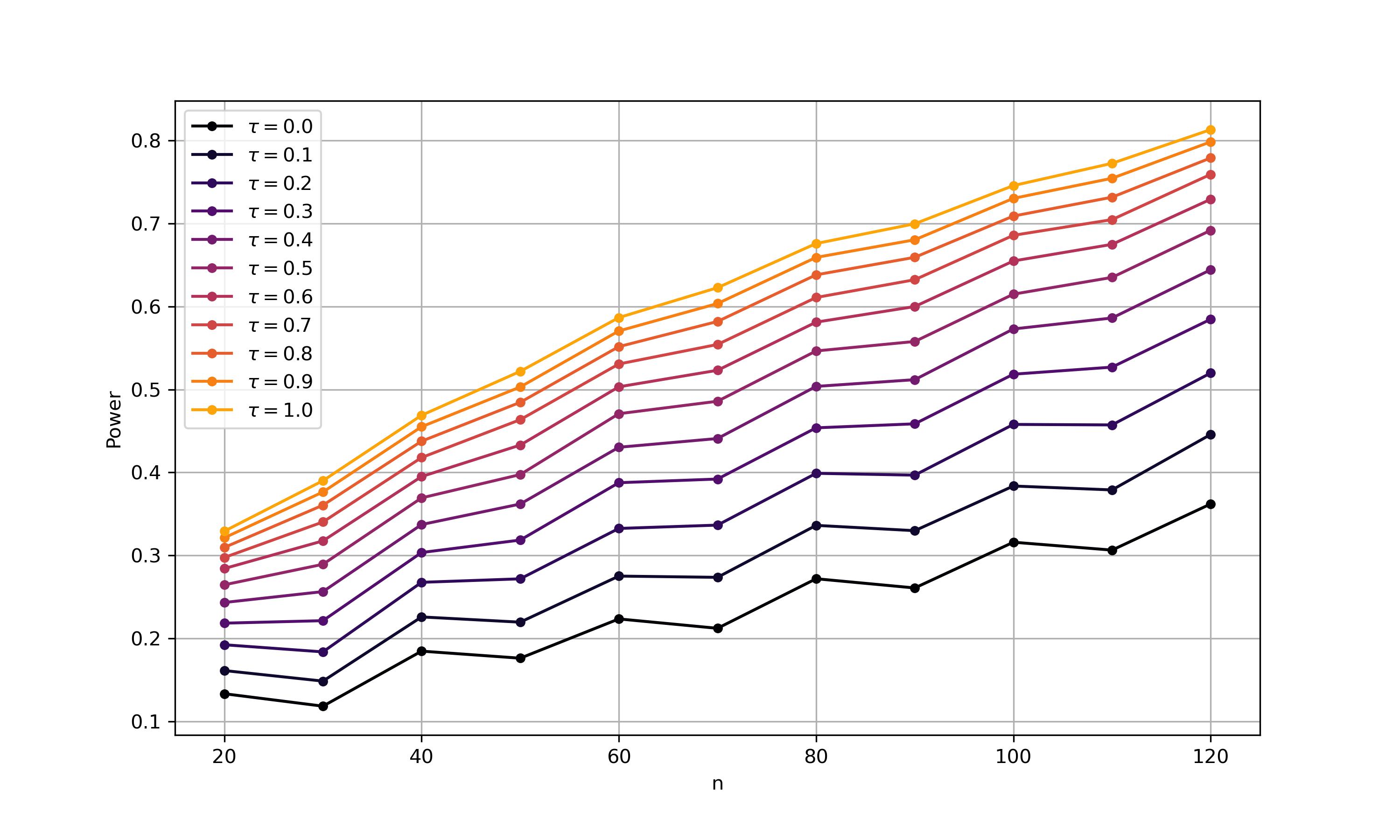}
		\subcaption{$15\%$ contamination}
	\end{subfigure}
	\caption{Empirical power of the Wald-type test statistics under increasing sample size for different values of $\tau$ in the absence of contamination (left) and with a $15\%$ contamination proportion.}
	\label{fig:powern}
\end{figure}

\begin{figure}[htb]
	\begin{subfigure}[c]{0.53\textwidth}
		\includegraphics[scale=0.28]{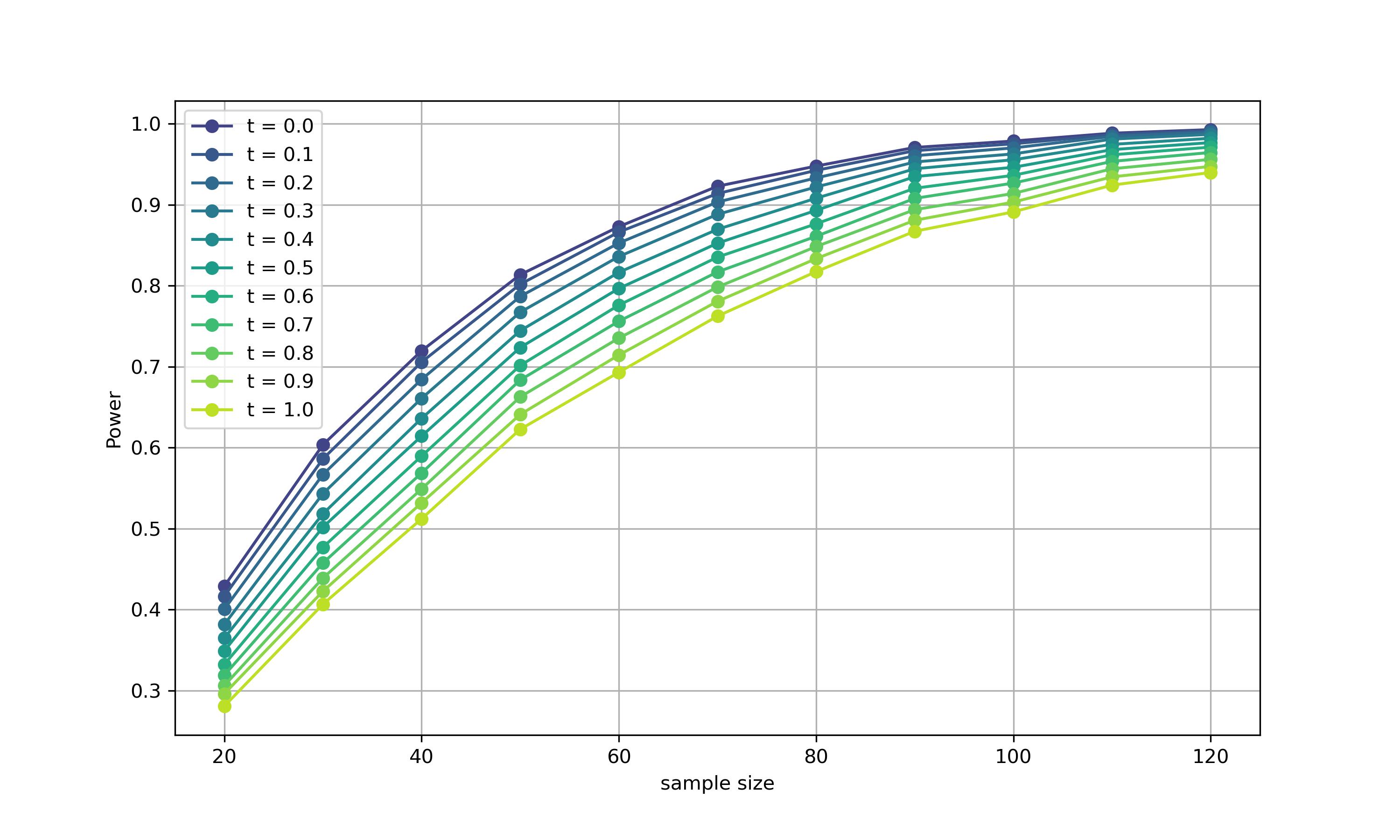}
		\subcaption{No contamination}
	\end{subfigure}
	\begin{subfigure}[c]{0.48\textwidth}
		\includegraphics[scale=0.28]{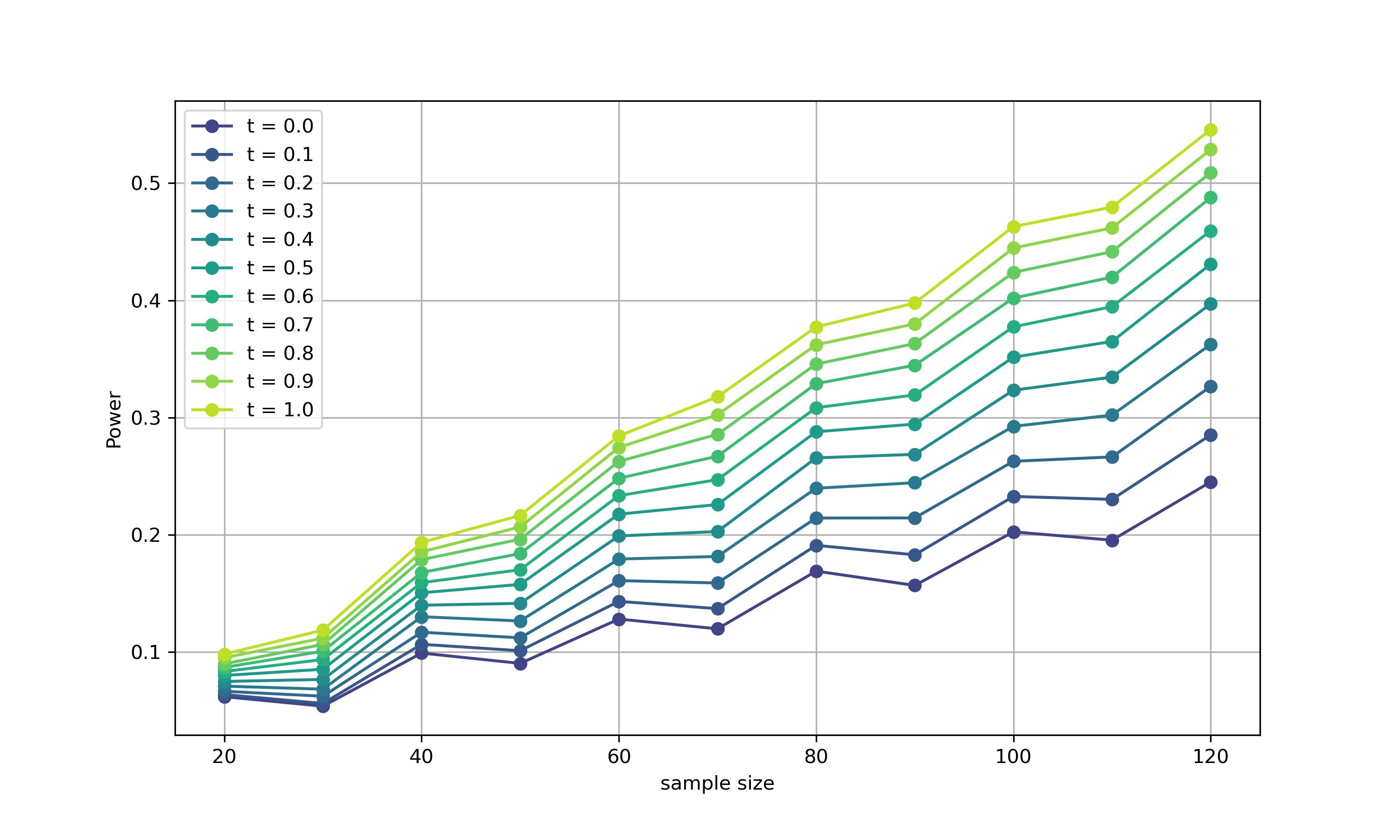}
		\subcaption{$15\%$ contamination}
	\end{subfigure}
	\caption{Empirical power of the Rao-type test statistics under increasing sample size for different values of $\tau$ in the absence of contamination (left) and with a $15\%$ contamination proportion.}
	\label{fig:powern_rao}
\end{figure}

Finally, figures \ref{fig:powercont} and \ref{fig:powercont_rao} represent the effect of contamination on the empirical power of the test for two sample sizes, $n=50$ (left) and $n=100$ (right). The strong impact of contamination on the MLE is particularly notable, as its empirical power drops drastically from nearly 1 in the absence of contamination to values below 0.1 for a contamination proportion near $20\%.$ This indicates that the classical test is highly sensitive to contamination, significantly altering its decisions. In contrast, Wald-type and Rao-type test statistic based on robust MDPDE keeps competitive, rejecting the null hypothesis with a quite high empirical power even under heavy contamination.

\begin{figure}[htb]
	\begin{subfigure}[c]{0.53\textwidth}
		\includegraphics[scale=0.28]{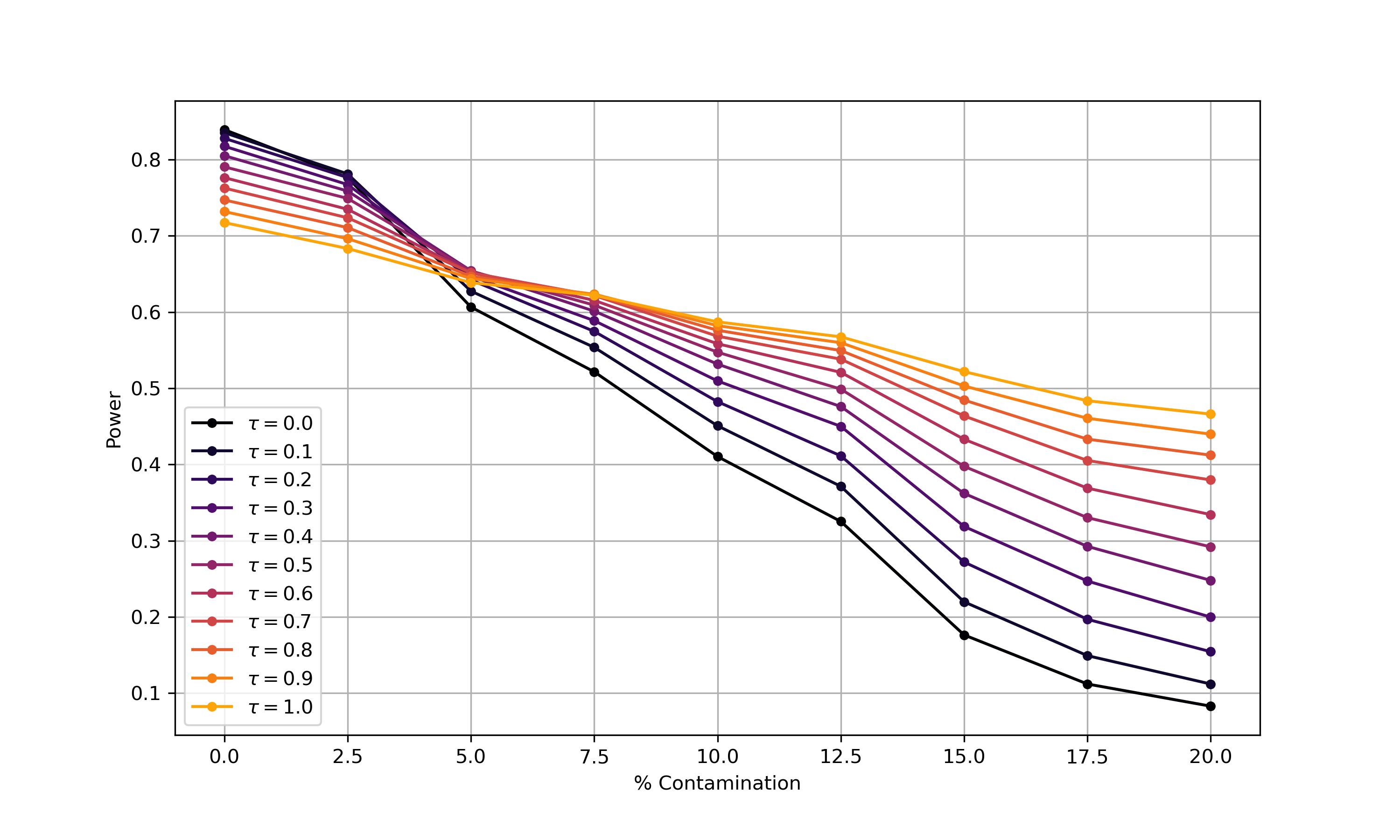}
		\subcaption{$n=50$}
	\end{subfigure}
	\begin{subfigure}[c]{0.48\textwidth}
		\includegraphics[scale=0.28]{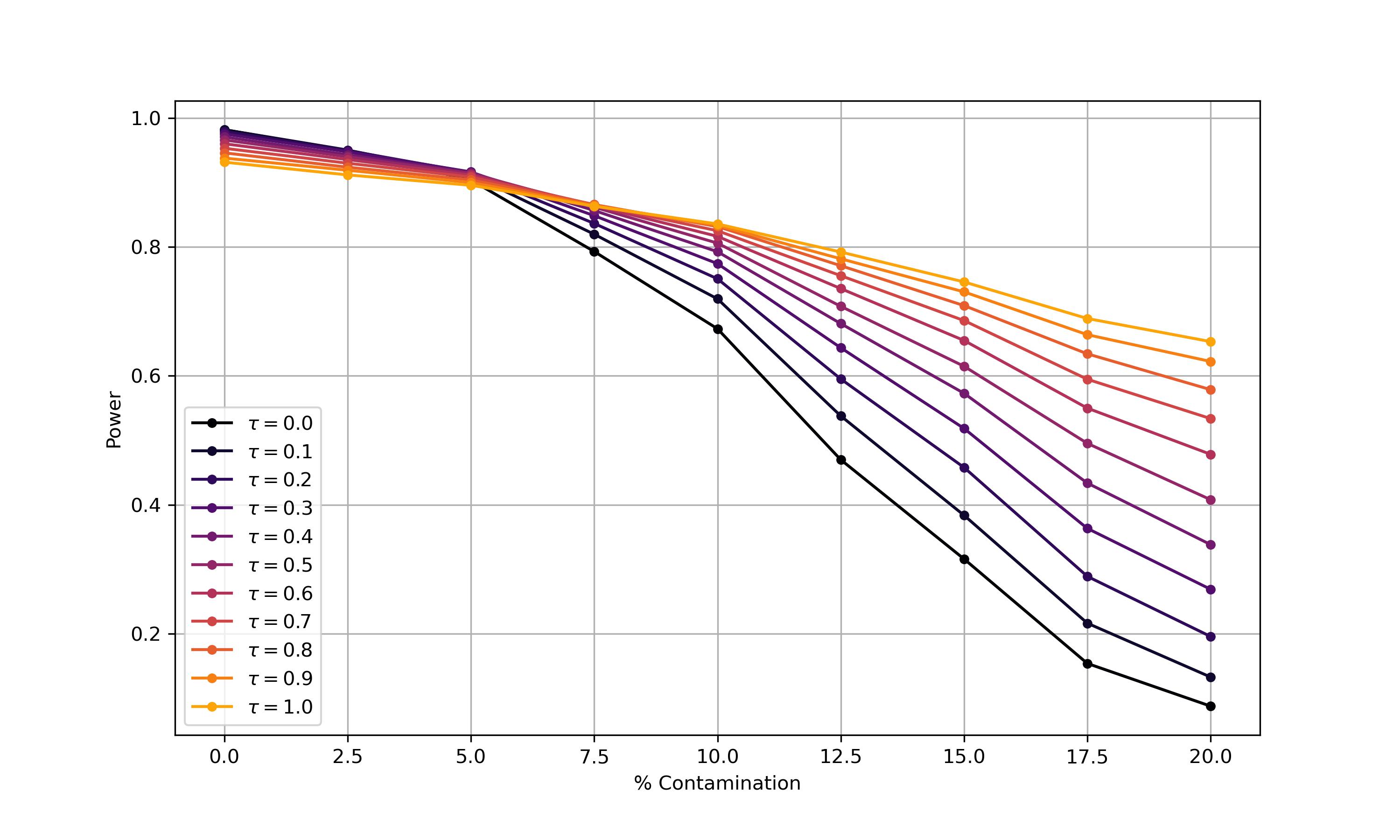}
		\subcaption{$n=100$}
	\end{subfigure}
	\caption{Empirical power of the Wald-type test statistics with $n=50$ (left) and $n=100$ (rigth) under increasing contamination proportion under different values of $\tau$.}
	\label{fig:powercont}
\end{figure}

\begin{figure}[htb]
	\begin{subfigure}[c]{0.53\textwidth}
		\includegraphics[scale=0.28]{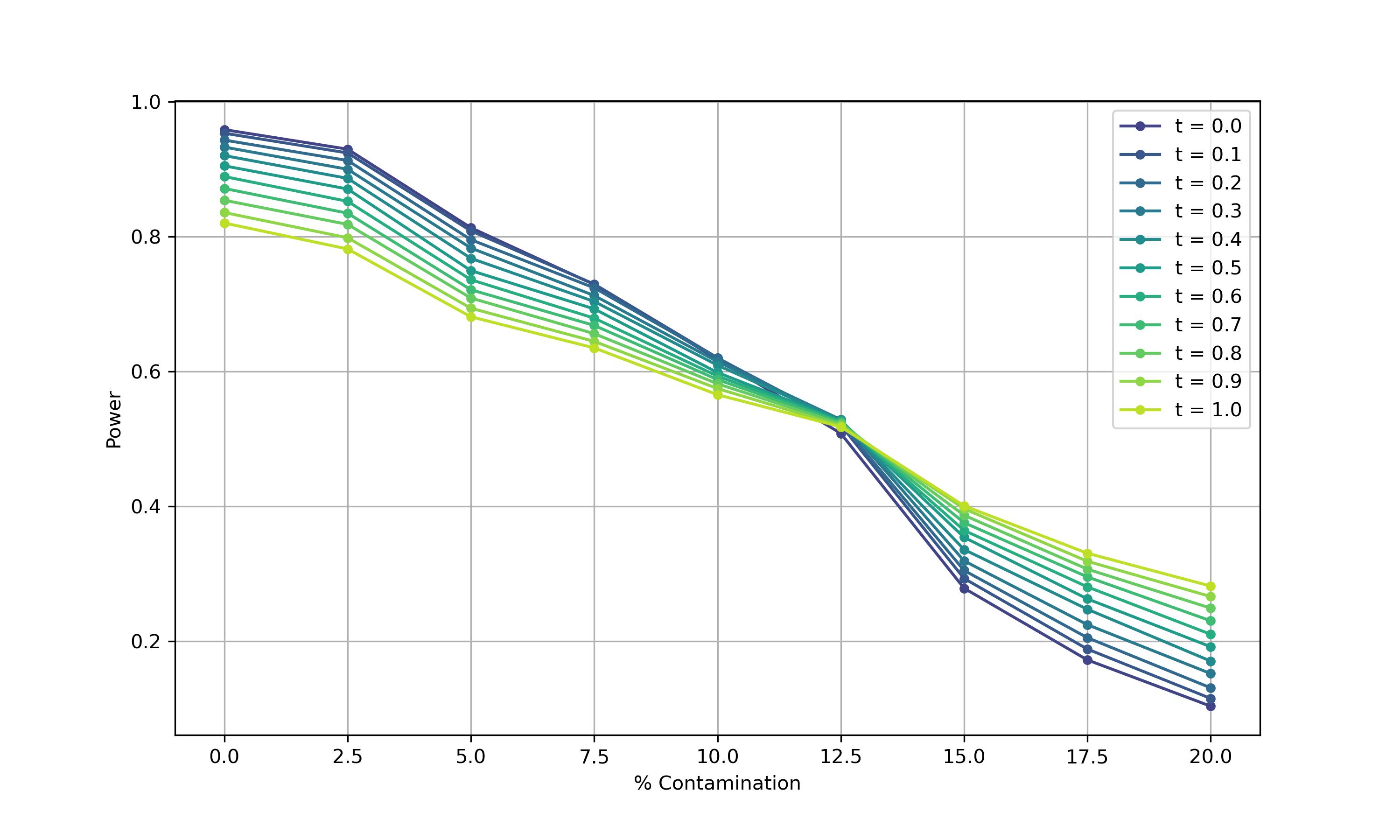}
		\subcaption{$n=50$}
	\end{subfigure}
	\begin{subfigure}[c]{0.48\textwidth}
		\includegraphics[scale=0.28]{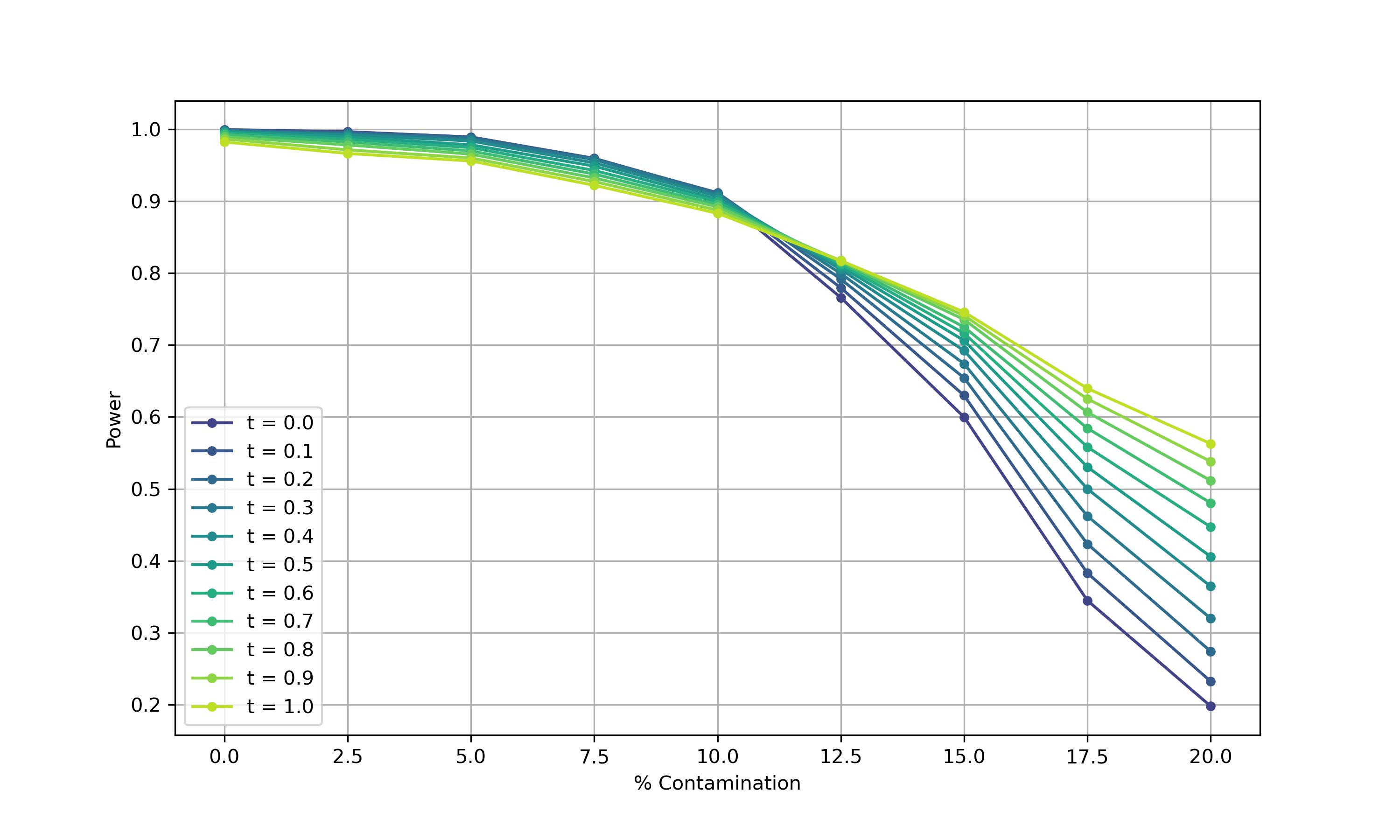}
		\subcaption{$n=100$}
	\end{subfigure}
	\caption{Empirical power of the Rao-type test statistics with $n=50$ (left) and $n=100$ (rigth) under increasing contamination proportion under different values of $\tau$.}
	\label{fig:powercont_rao}
\end{figure}


\section*{Acknowledgements}

This work was supported by the Spanish Grant PID2021-124933NB-I00. The four authors are members of the  Interdisciplinary Mathematics Institute (IMI).

\bibliographystyle{abbrvnat}

\end{document}